\documentclass[a4paper]{article}
\usepackage{amsmath}
\usepackage{amsthm}
\usepackage{amssymb}
\usepackage{dsfont}
\usepackage{mathrsfs}
\usepackage{mathabx}
\usepackage{mathtools}
\usepackage{stmaryrd}
\usepackage{hyperref}
\usepackage[all]{xy}
\usepackage{tikz}
\usepackage{graphicx}
\usetikzlibrary{graphs}
\usetikzlibrary{calc}
\usetikzlibrary{math}

\title{Parallelism Theorem and Derived Rules for Parallel Coherent Transformations}

\author{Thierry  Boy de la Tour}

\date{Univ. Grenoble Alpes, CNRS, Grenoble INP, LIG \\ 38000 Grenoble,
  France \\
{\small \texttt{thierry.boy-de-la-tour@imag.fr} }}

 \newtheorem{theorem}{Theorem} [section]
 \newtheorem{lemma}[theorem]{Lemma}

 \newtheorem{definition}[theorem]{Definition}
 
\def\clap#1{\hbox to 0pt{\hss#1\hss}}

\definecolor{lblue}{rgb}{0.4,0.4,1}
\definecolor{lgray}{rgb}{0.8,0.8,0.8}

\newcommand{\tuple}[1]{(#1)}
\newcommand{\Nat}{\mathds{N}}

\newcommand{\ensvide}{\varnothing}

\newcommand{\defeq}{\stackrel{\mathrm{\scriptscriptstyle def}}{=}}
\newcommand{\invf}[1]{#1^{-1}}

\newcommand{\aCat}{\mathcal{C}}

\newcommand{\dirtrans}[2]{\Delta(#1,#2)}
\newcommand{\dirtranspo}[2]{\Delta^{\mathrm{PO}}(#1,#2)}

\newcommand{\adt}{\gamma}
\newcommand{\anassocdt}{\delta}
\newcommand{\weakspan}{weak span}
\newcommand{\arule}{\rho}
\newcommand{\aspan}{\sigma}
\newcommand{\R}{\mathcal{R}}
\newcommand{\PO}{{\footnotesize PO}}
\newcommand{\id}[1]{\mathrm{id}_{#1}}
\newcommand{\trans}[1]{\xRightarrow{\,#1\,}}
\newcommand{\assoc}[1]{\check{#1}}

\newcommand{\M}{\mathcal{M}}

\newcommand{\Src}{\mathcal{S}}
\newcommand{\ruletrans}[1]{\mathrel{\Vdash_{#1}}}
\newcommand{\ruledir}[1]{\mathrel{\vdash_{#1}}}

\newcommand{\cubenodes}[8]{
\node (LL) at (0,0) {$#1$}; \node (LF) at (4,-1) {$#2$};
\node (LB) at (3,1) {$#3$}; \node (LR) at (7,0) {$#4$};
\node (UL) at (0,3.2) {$#5$}; \node (UF) at (4,2.2) {$#6$};
\node (UB) at (3,4.2) {$#7$}; \node (UR) at (7,3.2) {$#8$};}

\begin{document}
\maketitle
\begin{abstract}
  An Independent Parallelism Theorem is proven in the theory of
  adhesive HLR categories. It shows the bijective correspondence
  between sequential independent and parallel independent direct
  derivations in the Weak Double-Pushout framework, see
  \cite{BoydelatE19}.  The parallel derivations are expressed by means
  of Parallel Coherent Transformations (PCTs), hence without assuming
  the existence of coproducts compatible with $\M$ as in the standard
  Parallelism Theorem.  It is aslo shown that a derived rule can be
  extracted from any PCT, in the sense that to any direct derivation
  of this rule corresponds a valid PCT.
 \end{abstract}

\section{Definitions}

We use a number of definitions and results from \cite{AdamekHS04},
that we recall (and simplify) here.

A (finite) \emph{source} in a category $\aCat$ is a pair
$\Src=\tuple{A,(f_i)_{i=1}^p}$ with an object $A$ and morphisms
$f_i:A\rightarrow A_i$ for $1\leq i\leq p$ (where $p\in\Nat$); $A$ is
the domain and $(A_i)_{i=1}^p$ the codomain of $\Src$. For any
morphism $f:B\rightarrow A$ we write $\Src\circ f$ for the source
$\tuple{B,(f_i\circ f)_{i=1}^p}$. $\Src$ is a \emph{mono-source} if
for all pair $f,g:B\rightarrow A$, the equation
$\Src\circ f = \Src\circ g$ (i.e., $f_i\circ f = f_i\circ g$ for all
$1\leq i\leq p$) implies $f=g$. A mono-source $\Src$ is
\emph{extremal} if for every source $\Src'$ and epimorphism $e$ such
that $\Src=\Src'\circ e$, then $e$ is an isomorphism \cite[Definitions
10.1-11]{AdamekHS04}. The notions of \emph{sink, epi-sink} and
\emph{extremal epi-sink} are dual to these \cite[10.63]{AdamekHS04}.

A source $\tuple{A,(f_i)_{i=1}^p}$ with codomain $(B_i)_{i=1}^p$ is
\emph{natural} for a sink $\tuple{(g_i)_{i=1}^p,C}$ with domain
$(B_i)_{i=1}^p$ if $g_i\circ f_i = g_j\circ f_j$ for all
$1\leq i,j\leq p$. A \emph{limit} of $\tuple{(g_i)_{i=1}^p,C}$ is a
source $\Src$ natural for $\tuple{(g_i)_{i=1}^p,C}$ such that for all
source $\Src'$ natural for $\tuple{(g_i)_{i=1}^p,C}$ there exists a
unique morphism $f$ such that $\Src = \Src'\circ f$. A limit of
$\tuple{g_1,g_2,C}$ is called a \emph{pullback} \cite[Definitions
11.3-8]{AdamekHS04}. The notions of \emph{natural} sink,
\emph{colimit} and \emph{pushout} are dual to these
\cite[11.27-28]{AdamekHS04}.

A limit is essentially unique in the sense that, if $\Src$ is a limit
of a sink, then any natural source for this sink is a limit iff it is
of the form $\Src\circ h$ where $h$ is an isomorphism
\cite[11.7]{AdamekHS04}. Besides, every limit is an extremal
mono-source \cite[11.6]{AdamekHS04}. By duality, colimits are
essentially unique and every colimit is an extremal epi-sink
\cite[11.29]{AdamekHS04}.

If the following diagram commutes,
\begin{center}
  \begin{tikzpicture}[scale=1.5]
        \node (G) at (0,0) {$C$}; \node (L) at (0,1) {$A$};
        \node (K) at (1,1) {$B$}; \node (D) at (1,0) {$D$};
        \node (RK) at (2,1) {$E$}; \node (H) at (2,0) {$F$}; 
        \node at (0.5,0.5) {(1)}; \node at (1.5,0.5) {(2)};
        \path[<-] (K) edge (L);
        \path[->] (L) edge (G);
        \path[->] (K) edge (D);
        \path[<-] (D) edge (G);
        \path[->] (D) edge (H);
        \path[->] (K) edge (RK);
        \path[->] (RK) edge (H);
  \end{tikzpicture}
\end{center}
then the outer rectangle is a pushout whenever (1) and (2) are both
pushouts (\emph{pushout composition}), and if (1) and the outer
rectangle are pushouts then so is (2) (\emph{pushout decomposition}),
see \cite[11.10]{AdamekHS04} for the dual result.

\section{Parallel Coherent Transformations}

The material in this section is taken from \cite{BoydelatE19}, with
some modifications.

\begin{definition}\label{def-weakspan}
  A \emph{\weakspan} $\arule$ is a diagram
  $L\xleftarrow{l} K \xleftarrow{i} I \xrightarrow{r} R$ in
  $\aCat$. Given an object $G$ of $\aCat$ and a \weakspan\ $\arule$, a
  \emph{direct transformation $\adt$ of $G$ by $\arule$} is a diagram
\begin{center}
  \begin{tikzpicture}[xscale=1.8, yscale=1.5]
    \node (L) at (0,1) {$L$}; \node (K) at (1,1) {$K$};  \node (I) at
    (2,1) {$I$};  \node (R) at (3,1) {$R$}; \node (G) at (0.5,0) {$G$};
    \node (D) at (1.5,0) {$D$}; \node (H) at (2.5,0) {$H$};
    \node at (0.75,0.5) {=}; \node at (2.25,0.5) {\PO};
    \path[->] (K) edge node[fill=white, font=\footnotesize] {$l$} (L);
    \path[->] (I) edge node[fill=white, font=\footnotesize] {$i$} (K);
    \path[->] (I) edge node[fill=white, font=\footnotesize] {$r$} (R);
    \path[->] (L) edge node[fill=white, font=\footnotesize] {$m$} (G);
    \path[->] (K) edge node[fill=white, font=\footnotesize] {$k$} (D);
    \path[->] (D) edge node[fill=white, font=\footnotesize] {$f$} (G);
    \path[->] (I) edge node[fill=white, font=\footnotesize] {$k\circ i$} (D);
    \path[->] (D) edge node[fill=white, font=\footnotesize] {$g$} (H);
    \path[->] (R) edge node[fill=white, font=\footnotesize] {$n$} (H);
  \end{tikzpicture}
\end{center}
such that $f\circ k = m\circ l$ and
$\tuple{g,n,H}$ is a pushout of $\tuple{I,r,k\circ i}$; we then write
$G\trans{\adt} H$. Let
$\dirtrans{G}{\arule}$ be the set of all direct transformations of $G$
by $\arule$. For a set $\R$ of {\weakspan}s, let
$\dirtrans{G}{\R}\defeq \biguplus_{\arule\in\R}\dirtrans{G}{\arule}$.

If $\tuple{f,m,G}$ is a pushout of $\tuple{K,l,k}$, then $\adt$ is
called \emph{Weak Double-Pushout}. Let $\dirtranspo{G}{\R}$ be the set
of Weak Double-Pushouts in $\dirtrans{G}{\R}$.
\end{definition}
As $\arule$ is part of any diagram $\adt\in\dirtrans{G}{\arule}$, it
is obvious that
$\dirtrans{G}{\arule}\cap \dirtrans{G}{\arule'} = \ensvide$ whenever
$\arule\neq \arule'$. A span is of course a weak span where $I=K$ and
$i=\id{K}$, and in this case a Weak Double-Pushout is a standard
Double-Pushout diagram.

In the rest of the paper, when we refer to some weak span $\arule$,
possibly indexed by a natural number, we will also assume the objects
and morphisms $L$, $K$, $I$, $R$, $l$, $i$ and $r$, indexed by the
same number, as given in the definition of weak spans. The same scheme
will be used for direct transformations and indeed for all diagrams
given in future definitions.

\begin{definition}\label{def-coherent}
  Given an object $G$ of $\aCat$, two {\weakspan}s $\arule_1$ and
  $\arule_2$, and direct transformations $\adt_1\in
  \dirtrans{G}{\arule_1}$ and $\adt_2\in \dirtrans{G}{\arule_2}$, if
  there exist two morphisms $j_1^2: I_1\rightarrow D_2$ and $j_2^1: I_2\rightarrow
  D_1$ such that the diagram
  \begin{center}
  \begin{tikzpicture}[xscale=1.65, yscale=1.8]
    \node (L) at (0.5,1) {$L_2$}; \node (K) at (1.5,1) {$K_2$};  \node (I) at
    (2.5,1) {$I_2$};  \node (R) at (3.5,1) {$R_2$}; \node (G) at (0,0) {$G$};
    \node (D) at (2,0) {$D_2$}; \node (H) at (3,0) {$H_2$};
    \path[->] (K) edge node[fill=white, font=\footnotesize] {$l_2$} (L);
    \path[->] (I) edge node[fill=white, font=\footnotesize] {$i_2$} (K);
    \path[->] (I) edge node[fill=white, font=\footnotesize] {$r_2$} (R);
    \path[->] (L) edge node[fill=white, font=\footnotesize, near start] {$m_2$} (G);
    \path[->] (K) edge node[fill=white, font=\footnotesize] {$k_2$} (D);
    \path[->] (D) edge node[fill=white, font=\footnotesize] {$f_2$} (G);
    \path[->] (I) edge node[fill=white, font=\footnotesize] {$k_2\circ i_2$} (D);
    \path[->] (D) edge node[fill=white, font=\footnotesize] {$g_2$} (H);
    \path[->] (R) edge node[fill=white, font=\footnotesize] {$n_2$} (H);
    \node (L1) at (-0.5,1) {$L_1$}; \node (K1) at (-1.5,1) {$K_1$};  \node (I1) at
    (-2.5,1) {$I_1$};  \node (R1) at (-3.5,1) {$R_1$};
    \node (D1) at (-2,0) {$D_1$}; \node (H1) at (-3,0) {$H_1$};
    \path[->] (K1) edge node[fill=white, font=\footnotesize] {$l_1$} (L1);
    \path[->] (I1) edge node[fill=white, font=\footnotesize] {$i_1$} (K1);
    \path[->] (I1) edge node[fill=white, font=\footnotesize] {$r_1$} (R1);
    \path[->] (L1) edge node[fill=white, font=\footnotesize, near start] {$m_1$} (G);
    \path[->] (K1) edge node[fill=white, font=\footnotesize] {$k_1$} (D1);
    \path[->] (D1) edge node[fill=white, font=\footnotesize] {$f_1$} (G);
    \path[->] (I1) edge node[fill=white, font=\footnotesize] {$k_1\circ i_1$} (D1);
    \path[->] (D1) edge node[fill=white, font=\footnotesize] {$g_1$} (H1);
    \path[->] (R1) edge node[fill=white, font=\footnotesize] {$n_1$} (H1);
  \path[-] (I1) edge[draw=white, line width=3pt]  (D);
  \path[-] (I) edge[draw=white, line width=3pt]  (D1);
    \path[->,dashed] (I1) edge node[fill=white, font=\footnotesize,
    near start] {$j_1^2$} (D);
    \path[->,dashed] (I) edge node[fill=white, font=\footnotesize,
    near start] {$j_2^1$} (D1);
  \end{tikzpicture}    
  \end{center}
  commutes, i.e., $f_2\circ j_1^2 = f_1\circ k_1\circ i_1$ and
  $f_1\circ j_2^1 = f_2\circ k_2\circ i_2$, then we say that $\adt_1$
  and $\adt_2$ are \emph{parallel coherent}.

  A \emph{parallel coherent diagram for $G$} is a commuting diagram
  $\Gamma$ in $\aCat$ constituted of diagrams
  $\adt_1,\dotsc,\adt_p\in\dirtrans{G}{\R}$ for some $p\geq 1$, and
  morphisms $j_a^b:I_a\rightarrow D_b$ for all $1\leq a,b\leq p$.
\end{definition}

Note that for any $\adt\in\dirtrans{G}{\R}$, the diagram $\adt$
extended with $j=k\circ i$ is parallel coherent. For any parallel
coherent diagram $\Gamma$, it is obvious that $\adt_a$ and $\adt_b$
are parallel coherent for all $1\leq a,b\leq p$, and that
\begin{center}
  \begin{tikzpicture} [xscale=2, yscale=1]
    \node (G) at (0,0) {$G$};
    \node (D1) at (1,1) {$D_1$};
    \node (Dn) at (1,-1) {$D_p$};
    \node (Ic) at (2,0) {$I_a$};
    \node at (1,0.1) {$\vdots$};
    \path[->] (D1) edge node[fill=white, font=\footnotesize] {$f_1$} (G);
    \path[->] (Dn) edge node[fill=white, font=\footnotesize] {$f_p$} (G);
    \path[->] (Ic) edge node[fill=white, font=\footnotesize] {$j_a^1$} (D1);
    \path[->] (Ic) edge node[fill=white, font=\footnotesize] {$j_a^p$} (Dn);
  \end{tikzpicture}
\end{center}
is a sub-diagram of $\Gamma$ for all $1\leq a\leq p$, hence commutes.

We can now consider the parallel transformation of an object by
parallel coherent diagram. The principle is that anything that is
removed by some direct transformation should be removed from the
input $G$, and all right hand sides should occur in the result.

\begin{definition}\label{def-transfo}
  For any object $G$ of $\aCat$ and $\Gamma$ a parallel coherent
  diagram for $G$, a \emph{parallel coherent transformation (or PCT)
    of $G$ by $\Gamma$} is a diagram as in Figure \ref{fig-pct} where:
  \begin{figure}[t]
    \centering
      \begin{tikzpicture}[xscale=2, yscale=0.9]
  \node (G) at (0,0) {$G$};
  \node (D) at (2,0) {$C$};
  \node (H) at (4,0) {$H$};
  \node at (1,0.1) {$\vdots$};
  \node at (3,0.1) {$\vdots$};
  \node (L1) at (0,2) {$L_1$};
  \node (K1) at (1,3) {$K_1$};
  \node (D1) at (1,1) {$D_1$};
  \node (I1) at (2,3) {$I_1$};
  \node (R1) at (3,4) {$R_1$};
  \node (H1) at (3,1) {$F_1$};
  \node (Ln) at (0,-2) {$L_p$};
  \node (Kn) at (1,-3) {$K_p$};
  \node (Dn) at (1,-1) {$D_p$};
  \node (In) at (2,-3) {$I_p$};
  \node (Rn) at (3,-4) {$R_p$};
  \node (Hn) at (3,-1) {$F_p$};
  \node at (2.5,2){\PO}; \node at (2.5,-2) {\PO};
  \path[->] (K1) edge node[fill=white, font=\footnotesize] {$l_1$} (L1) ;
  \path[->] (L1) edge node[fill=white, font=\footnotesize] {$m_1$} (G);
  \path[->] (K1) edge node[fill=white, font=\footnotesize] {$k_1$} (D1);
  \path[->] (D1) edge node[fill=white, font=\footnotesize] {$f_1$} (G);
  \path[->] (I1) edge node[fill=white, font=\footnotesize] {$i_1$} (K1);
  \path[->] (I1) edge node[fill=white, font=\footnotesize] {$r_1$} (R1);
  \path[->] (R1) edge node[fill=white, font=\footnotesize] {$o_1$} (H1);
  \path[->] (D) edge node[fill=white, font=\footnotesize, near start] {$e_1$} (D1);
  \path[->] (D) edge node[fill=white, font=\footnotesize] {$s_1$} (H1);
  \path[->] (H1) edge node[fill=white, font=\footnotesize] {$h_1$} (H);
  \path[->,dashed] (I1) edge node[fill=white, font=\footnotesize] {$d_1$} (D);
  \path[->] (Kn) edge node[fill=white, font=\footnotesize] {$l_p$} (Ln) ;
  \path[->] (Ln) edge node[fill=white, font=\footnotesize] {$m_p$} (G);
  \path[->] (Kn) edge node[fill=white, font=\footnotesize] {$k_p$} (Dn);
  \path[->] (Dn) edge node[fill=white, font=\footnotesize] {$f_p$} (G);
  \path[->] (In) edge node[fill=white, font=\footnotesize] {$i_p$} (Kn);
  \path[->] (In) edge node[fill=white, font=\footnotesize] {$r_p$} (Rn);
  \path[->] (Rn) edge node[fill=white, font=\footnotesize] {$o_p$} (Hn);
  \path[->] (D) edge node[fill=white, font=\footnotesize, near start] {$e_p$} (Dn);
  \path[->] (D) edge node[fill=white, font=\footnotesize] {$s_p$} (Hn);
  \path[->] (Hn) edge node[fill=white, font=\footnotesize] {$h_p$} (H);
  \path[->,dashed] (In) edge node[fill=white, font=\footnotesize] {$d_p$} (D);
  \path[->] (I1) edge node[fill=white, font=\footnotesize] {$j^1_1$} (D1);
  \path[->] (In) edge node[fill=white, font=\footnotesize] {$j^p_p$} (Dn);
  \path[-] (In) edge[draw=white, line width=3pt]  (D1);
  \path[->] (In) edge node[fill=white, font=\footnotesize, near start] {$j_p^1$} (D1);
  \path[-] (I1) edge[draw=white, line width=3pt]  (Dn);
  \path[->] (I1) edge node[fill=white, font=\footnotesize, near start] {$j^p_1$} (Dn);
  \node at (1.5,0){\footnotesize limit}; \node at (3.5,0) {\footnotesize colimit}; 
\end{tikzpicture}
\caption{A parallel coherent transformation (or PCT)}
\label{fig-pct}
\end{figure}
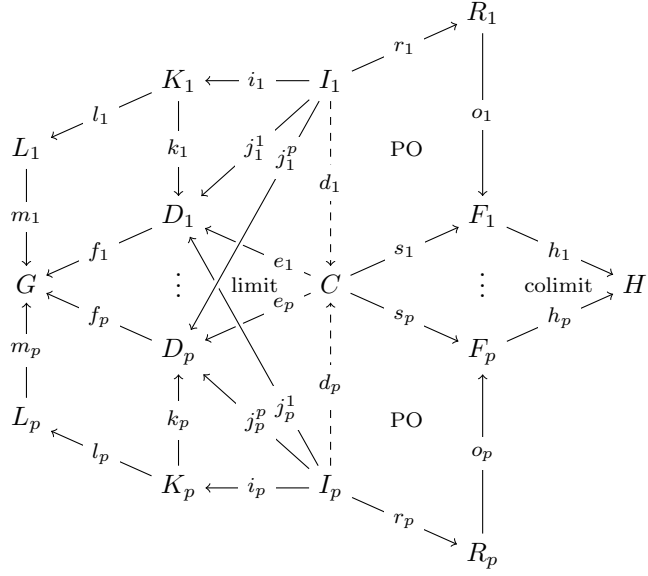

\begin{itemize}
\item $\tuple{C, e_1,\dotsc,e_p}$ is a limit of
  $\tuple{f_1,\dotsc,f_p,G}$,
\item for all $1\leq c\leq p$, $d_c:I_c\rightarrow C$ is the unique
  morphism such that for all $1\leq a\leq p$, $j_c^{a} = e_{a}\circ d_c$,
\item  for all $1\leq a\leq p$, $\tuple{s_a,o_a, F_a}$ is a pushout of
$\tuple{I_a,r_a,d_a}$,
\item $\tuple{h_1,\dotsc,h_p,H}$ is a colimit of $\tuple{C,s_1,\dotsc,s_p}$.
\end{itemize}
If such a diagram exists we write $G\trans{\Gamma} H$.
\end{definition}

See \cite{BoydelatE19} for examples.

We now assume a class of monomorphisms $\M$ of $\aCat$
that confers $\tuple{\aCat,\M}$ a structure of adhesive HLR
category. We do not give here the rather long definition of this
concept, which can be found in \cite{EhrigEPT06}. In the results below
we use the following properties of adhesive HLR categories.
\begin{enumerate}
\item $\M$ contains all isomorphisms, is closed under composition and
  under decomposition, i.e., if $g\circ f\in\M$ and $g\in\M$ then
  $f\in\M$.
\item $\M$ is closed under pushouts and pullbacks, i.e., if a square
  \begin{center}
    \begin{tikzpicture}[scale=1.5]
      \node (A) at (0,1) {$A$}; \node (B) at (1,1) {$B$};
      \node (C) at (0,0) {$C$}; \node (D) at (1,0) {$D$};
      \path[->] (A) edge node[fill=white, font=\footnotesize] {$f$} (B); 
      \path[->] (A) edge  (C); 
      \path[->] (B) edge (D); 
      \path[->] (C) edge node[fill=white, font=\footnotesize] {$f'$} (D); 
    \end{tikzpicture}
  \end{center}
  is a pushout and $f\in\M$ then $f'\in\M$; if it is a pullback and
  $f'\in\M$ then $f\in\M$.
\item Every pushout along a $\M$-morphism is a pullback.
\item The $\M$-POPB decomposition lemma: in the commuting diagram
  \begin{center}
        \begin{tikzpicture}[scale=1.5]
      \node (G) at (0,0) {$C$};
      \node (L) at (0,1) {$A$};
      \node (K) at (1,1) {$B$};
      \node (D) at (1,0) {$D$};
      \node (RK) at (2,1) {$E$};
      \node (H) at (2,0) {$F$}; 
      \path[<-] (K) edge node[fill=white, font=\footnotesize] {$u$} (L);
      \path[->] (L) edge node[fill=white, font=\footnotesize] {$v$} (G);
      \path[->] (K) edge (D);
      \path[<-] (D) edge (G);
      \path[->] (D) edge node[fill=white, font=\footnotesize] {$w$} (H);
      \path[->] (K) edge (RK);
      \path[->] (RK) edge (H);
    \end{tikzpicture}
  \end{center}
  if the outer square is a pushout, the right square a pullback,
  $w\in\M$ and ($u\in\M$ or $v\in\M$), then the left and right squares
  are both pushouts and pullbacks.
\item The cube POPB lemma: in the commuting diagram
  \begin{center}
    \begin{tikzpicture}[scale=0.5]
  \cubenodes{D'}{B'}{C'}{A'}{D}{B}{C}{A};
  \path[<-] (LL) edge  (LF);
  \path[<-] (LF) edge (LR);
  \path[->] (LR) edge (LB);
  \path[->] (LB) edge (LL);
  \path[<-] (LL) edge (UL); 
  \path[-] (LF) edge [draw=white, line width=3pt] (UF); 
  \path[<-] (LF) edge (UF); 
  \path[<-] (LR) edge (UR); 
  \path[<-] (LB) edge (UB); 
  \path[-] (UL) edge [draw=white, line width=3pt] (UF);
  \path[<-] (UL) edge (UF);
  \path[<-] (UF) edge (UR);
  \path[->] (UR) edge (UB);
  \path[->] (UB) edge (UL);
  \end{tikzpicture}
  \end{center}
  where all horizontal morphisms are in $\M$, the top face is a
  pullback and the left faces are pushouts, then the bottom face is a
  pullback iff the right faces are pushouts.
\end{enumerate}

In this theory, the morphisms in production rules are elements of
$\M$, and the direct derivations are Double Pushouts.
\begin{definition}\label{def-assoc}
  A \emph{$\M$-weak span}  $\arule$ is a weak span whose morphisms $l,i,r$
  belong to $\M$. The
  \emph{associated span $\assoc{\arule}$ of $\arule$} is the
  diagram $L\xleftarrow{l} K \xrightarrow{r'} R'$ where
  $\tuple{i',r',R'}$ is a pushout of $\tuple{I,i,r}$. 
\end{definition}

The associated span always exists and is a $\M$-span by the closure
properties of $\M$. This association is reflected in the following
equivalence of direct derivations.

\begin{lemma}\label{lm-assocspan}
  For all objects $G, H$ of $\aCat$ and $\M$-weak span $\arule$, we have
  \[\exists \adt\in\dirtranspo{G}{\arule} \text{ s.t. } G \trans{\adt}
    H \text{ iff }  \exists \anassocdt\in\dirtranspo{G}{\assoc{\arule}}
    \text{ s.t. } G \trans{\anassocdt}
    H.\] 
\end{lemma}
\begin{proof}
  Only if part. Since $r\in\M$ there exists a pushout $\tuple{g,n,H}$
  of $\tuple{I,r,k\circ i}$, then $n\circ r = g\circ k\circ i$,
  hence there is a unique morphism $n': RK\rightarrow H$ such that
  $n'\circ i'=n$ and $n'\circ r'=g\circ k$. By pushout decomposition
  $\tuple{g,n',H}$ is a pushout of $\tuple{K,r',k}$.
  \begin{center}
      \begin{tikzpicture}[scale=1.5]
        \node (G) at (0,0) {$G$};
        \node (L) at (0,1) {$L$};
        \node (K) at (1,1) {$K$};
        \node (D) at (1,0) {$D$};
        \node (I) at (1,2) {$I$};
        \node (R) at (2,2) {$R$};
        \node (RK) at (2,1) {$R'$};
        \node (H) at (2,0) {$H$}; 
        \path[->] (K) edge node[fill=white, font=\footnotesize] {$l$} (L);
        \path[->] (I) edge node[fill=white, font=\footnotesize] {$i$} (K);
        \path[->] (I) edge node[fill=white, font=\footnotesize] {$r$} (R);
        \path[->] (L) edge node[fill=white, font=\footnotesize] {$m$} (G);
        \path[->] (K) edge node[fill=white, font=\footnotesize] {$k$} (D);
        \path[->] (D) edge node[fill=white, font=\footnotesize] {$f$} (G);
        \path[->] (D) edge node[fill=white, font=\footnotesize] {$g$} (H);
        \path[->] (R) edge [bend left] node[fill=white, font=\footnotesize] {$n$} (H);
        \path[->] (R) edge node[fill=white, font=\footnotesize] {$i'$} (RK);
        \path[->] (K) edge node[fill=white, font=\footnotesize] {$r'$} (RK);
        \path[->] (RK) edge [dashed] node[fill=white, font=\footnotesize] {$n'$} (H);
      \end{tikzpicture}
  \end{center}
  If part. Since $r\in\M$ the $r'\in\M$ hence there exists a pushout
  $\tuple{g,n',H}$ of $\tuple{K,r',k}$, then by pushout composition
  $\tuple{g, n'\circ i',H}$ is a pushout of $\tuple{I,r,k\circ
    i}$. 
\end{proof}

This lemma suggests that weak spans can be analyzed with respect to
the properties of their associated spans, on which a wealth of results
is known. 

\section{Sequential and Parallel Independence}

\begin{definition}
  For any $\M$-weak span $\arule$, object $G$ and
  $\adt\in\dirtranspo{G}{\arule}$, let
  $\assoc{\adt}\in\dirtranspo{G}{\assoc{\arule}}$ be the diagram built
    from $\adt$ in the proof of Lemma \ref{lm-assocspan}.

  Given $\M$-weak spans $\arule_1$ and $\arule_2$, an object $G$ of $\aCat$
  and direct transformations $\adt_1\in\dirtranspo{G}{\arule_1}$ and
  $\adt_2\in\dirtranspo{G}{\arule_2}$, $\adt_1$ and $\adt_2$ are
  \emph{parallel independent} if $\assoc{\adt_1}$ and $\assoc{\adt_2}$
  are parallel independent, i.e., if there exist morphisms
  $j_1:L_1\rightarrow D_2$ and $j_2:L_2\rightarrow D_1$ such that
  $f_2\circ j_2= m_1$ and $f_1\circ j_2=m_2$.
  \begin{center}
  \begin{tikzpicture}[xscale=1.65, yscale=1.8]
    \node (L) at (0.5,1) {$L_2$}; \node (K) at (1.5,1) {$K_2$};  \node (I) at
    (2.5,1) {$I_2$};  \node (R) at (3.5,1) {$R_2$}; \node (G) at (0,0) {$G$};
    \node (D) at (2,0) {$D_2$}; \node (H) at (3,0) {$H_2$};
    \path[->] (K) edge node[fill=white, font=\footnotesize] {$l_2$} (L);
    \path[->] (I) edge node[fill=white, font=\footnotesize] {$i_2$} (K);
    \path[->] (I) edge node[fill=white, font=\footnotesize] {$r_2$} (R);
    \path[->] (L) edge node[fill=white, font=\footnotesize, near end] {$m_2$} (G);
    \path[->] (K) edge node[fill=white, font=\footnotesize] {$k_2$} (D);
    \path[->] (D) edge node[fill=white, font=\footnotesize] {$f_2$} (G);
    \path[->] (I) edge node[fill=white, font=\footnotesize] {$k_2\circ i_2$} (D);
    \path[->] (D) edge node[fill=white, font=\footnotesize] {$g_2$} (H);
    \path[->] (R) edge node[fill=white, font=\footnotesize] {$n_2$} (H);
    \node (L1) at (-0.5,1) {$L_1$}; \node (K1) at (-1.5,1) {$K_1$};  \node (I1) at
    (-2.5,1) {$I_1$};  \node (R1) at (-3.5,1) {$R_1$};
    \node (D1) at (-2,0) {$D_1$}; \node (H1) at (-3,0) {$H_1$};
    \path[->] (K1) edge node[fill=white, font=\footnotesize] {$l_1$} (L1);
    \path[->] (I1) edge node[fill=white, font=\footnotesize] {$i_1$} (K1);
    \path[->] (I1) edge node[fill=white, font=\footnotesize] {$r_1$} (R1);
    \path[->] (L1) edge node[fill=white, font=\footnotesize, near end] {$m_1$} (G);
    \path[->] (K1) edge node[fill=white, font=\footnotesize] {$k_1$} (D1);
    \path[->] (D1) edge node[fill=white, font=\footnotesize] {$f_1$} (G);
    \path[->] (I1) edge node[fill=white, font=\footnotesize] {$k_1\circ i_1$} (D1);
    \path[->] (D1) edge node[fill=white, font=\footnotesize] {$g_1$} (H1);
    \path[->] (R1) edge node[fill=white, font=\footnotesize] {$n_1$} (H1);
  \path[-] (L1) edge[draw=white, line width=3pt]  (D);
  \path[-] (L) edge[draw=white, line width=3pt]  (D1);
    \path[->,dashed] (L1) edge node[fill=white, font=\footnotesize] {$j_1$} (D);
    \path[->,dashed] (L) edge node[fill=white, font=\footnotesize] {$j_2$} (D1);
  \end{tikzpicture}    
  \end{center}
  Given direct transformations $\adt_1\in\dirtranspo{G}{\arule_1}$ such
  that $G\trans{\adt_1}H_1$ and
  $\adt_2\in\dirtranspo{H_1}{\arule_2}$, $\adt_1$ and $\adt_2$ are
  \emph{sequential independent} if $\assoc{\adt_1}$ and $\assoc{\adt_2}$
  are sequential independent, i.e., if there exist morphisms
  $j'_1:R'_1\rightarrow D_2$ and $j'_2:L_2\rightarrow D_1$ such that
  $f_2\circ j'_1= n_1$ and $g_1\circ j'_2=m_2$.
  \begin{center}
  \begin{tikzpicture}[xscale=1.65, yscale=1.8]
    \node (L) at (0.5,1) {$L_2$}; \node (K) at (1.5,1) {$K_2$};  \node (I) at
    (2.5,1) {$I_2$};  \node (R) at (3.5,1) {$R_2$}; \node (G) at (0,0) {$H_1$};
    \node (D) at (2,0) {$D_2$}; \node (H) at (3,0) {$H_2$};
    \path[->] (K) edge node[fill=white, font=\footnotesize] {$l_2$} (L);
    \path[->] (I) edge node[fill=white, font=\footnotesize] {$i_2$} (K);
    \path[->] (I) edge node[fill=white, font=\footnotesize] {$r_2$} (R);
    \path[->] (L) edge node[fill=white, font=\footnotesize, near end] {${m_2}$} (G);
    \path[->] (K) edge node[fill=white, font=\footnotesize] {$k_2$} (D);
    \path[->] (D) edge node[fill=white, font=\footnotesize] {$f_2$} (G);
    \path[->] (I) edge node[fill=white, font=\footnotesize] {$k_2\circ i_2$} (D);
    \path[->] (D) edge node[fill=white, font=\footnotesize] {$g_2$} (H);
    \path[->] (R) edge node[fill=white, font=\footnotesize] {$n_2$} (H);
    \node (Rp1) at (-0.5,1) {$R'_1$}; \node (K1) at (-1.5,1) {$K_1$};
    \node (I1) at (-1.5,2) {$I_1$};  \node (R1) at (-0.5,2) {$R_1$};
    \node (D1) at (-1.5,0) {$D_1$}; \node (H1) at (-2.5,0) {$G$};
    \node (L1) at (-2.5,1) {$L_1$};
    \path[->] (K1) edge node[fill=white, font=\footnotesize] {$r'_1$} (Rp1);
    \path[->] (R1) edge node[fill=white, font=\footnotesize] {$i'_1$} (Rp1);
    \path[->] (I1) edge node[fill=white, font=\footnotesize] {$i_1$} (K1);
    \path[->] (I1) edge node[fill=white, font=\footnotesize] {$r_1$} (R1);
    \path[->] (Rp1) edge node[fill=white, font=\footnotesize, near end] {$n_1$} (G);
    \path[->] (K1) edge node[fill=white, font=\footnotesize] {$k_1$} (D1);
    \path[->] (D1) edge node[fill=white, font=\footnotesize] {$g_1$} (G);
    \path[->] (K1) edge node[fill=white, font=\footnotesize] {$l_1$} (L1);
    \path[->] (D1) edge node[fill=white, font=\footnotesize] {$f_1$} (H1);
    \path[->] (L1) edge node[fill=white, font=\footnotesize] {$m_1$} (H1);
  \path[-] (L) edge[draw=white, line width=3pt]  (D1);
  \path[-] (Rp1) edge[draw=white, line width=3pt]  (D); 
    \path[->,dashed] (L) edge node[fill=white, font=\footnotesize] {$j'_2$} (D1); 
    \path[->,dashed] (Rp1) edge node[fill=white, font=\footnotesize] {$j'_1$} (D);
  \end{tikzpicture}    
  \end{center}

\end{definition}

It is obvious that if $\adt_1\in\dirtranspo{G}{\arule_1}$ and
$\adt_2\in\dirtranspo{G}{\arule_2}$ are parallel independent then they
are also parallel coherent, and therefore that there is a parallel
coherent diagram $\Gamma$ corresponding to $\adt_1$ and $\adt_2$ with
$j_1^2=j_1\circ l_1\circ i_1$ and $j_2^1 = j_2\circ l_2\circ i_2$, in
the sequel this $\Gamma$ will be written $\tuple{\adt_1,\adt_2}$.

\begin{theorem}[Independent Parallelism Theorem]\label{th-parallelism}
  For any $\M$-weak spans $\arule_1$, $\arule_2$, objects $G$, $H_1$,
  $H$ and direct transformation $\adt_1\in\dirtranspo{G}{\arule_1}$
  such that $G\trans{\adt_1}H_1$, then
  \begin{enumerate}
  \item (analysis) to any $\adt_2\in\dirtranspo{G}{\arule_2}$ such that
    $\adt_1,\adt_2$ are parallel independent and
    $G\trans{\tuple{\adt_1,\adt_2}} H$, we can associate a
    $\adt'_2\in\dirtranspo{H_1}{\arule_2}$ such that
    $G\trans{\adt_1} H_1\trans{\adt'_2} H$ is sequential
    independent,
  \item (synthesis) to any $\adt'_2\in\dirtranspo{H_1}{\arule_2}$ such that $G\trans{\adt_1}
    H_1\trans{\adt'_2} H$ is sequential independent we can
    associate a $\adt_2\in\dirtranspo{G}{\arule_2}$ such that
    $\adt_1,\adt_2$ are parallel independent and
    $G\trans{\tuple{\adt_1,\adt_2}} H$,
  \item and these two correspondences are inverse to each other up to isomorphism.
  \end{enumerate}
\end{theorem}
\begin{proof}
  \begin{enumerate}
  \item We consider a direct transformation $G\trans{\adt_2} H_2$ with
    morphisms $j_1: L_1\rightarrow D_2$, $j_2: L_2\rightarrow D_1$
    such that $f_2\circ j_1=m_1$, $f_1\circ j_2=m_2$, and a parallel
    coherent transformation $G\trans{\tuple{\adt_1,\adt_2}} H$ with 
    morphisms $j_1^2: I_1\rightarrow D_2$, $j^1_2: I_2\rightarrow D_1$
    such that $f_2\circ j^2_1=f_1\circ k_1\circ i_1$, $f_1\circ
    j^1_2=f_2\circ k_2\circ i_2$. Since pushout squares commute, we
    get $f_2\circ j^2_1= m_1\circ l_1\circ i_1 = f_2\circ j_1 \circ
    l_1\circ i_1$, and since $\M$-morphisms are stable by pushouts we
    have $f_2\in\M$, hence $f_2$ is a monomorphism so that $j^2_1= j_1 \circ
    l_1\circ i_1$. Symmetrically we get $f_1\in\M$ and $j^1_2= j_2 \circ
    l_2\circ i_2$; hence these two morphisms need not be depicted in
    Figure~\ref{fig-analysis}. 

    \begin{figure}[t]
      \centering
    \begin{tikzpicture}[scale=1.9]
        \node (G) at (0,0) {$G$};
        \node (D) at (1,-1) {$C$};
        \node (H) at (2,-2) {$H$};
        \node (L1) at (0,1) {$L_1$};
        \node (K1) at (1,1) {$K_1$};
        \node (D1) at (1,0) {$D_1$};
        \node (H1) at (2,0) {$H_1$};
        \node (I1) at (1,2) {$I_1$};
        \node (R1) at (2,2) {$R_1$}; 
        \node (Rp1) at (2,1) {$R'_1$}; 
        \node (Hp1) at (2,-1) {$F_1$}; 
        \node (L2) at (-1,0) {$L_2$};
        \node (K2) at (-1,-1) {$K_2$};
        \node (D2) at (0,-1) {$D_2$};
        \node (H2) at (0,-2) {$H_2$};
        \node (I2) at (-2,-1) {$I_2$};
        \node (R2) at (-2,-2) {$R_2$};
        \node (Hp2) at (1,-2) {$F_2$};
        \node at (1.6,1.5) {(1)};
        \node at (0.5,-0.5) {(2)}; \node at (1.6,-0.5) {(3)}; 
        \path[->] (K1) edge node[fill=white, font=\footnotesize] {$l_1$} (L1);
        \path[->] (I1) edge node[fill=white, font=\footnotesize] {$i_1$} (K1);
        \path[->] (I1) edge node[fill=white, font=\footnotesize] {$r_1$} (R1);
        \path[->] (L1) edge node[fill=white, font=\footnotesize, near start] {$m_1$} (G);
        \path[->] (K1) edge node[fill=white, font=\footnotesize] {$k_1$} (D1);
        \path[->] (D1) edge node[fill=white, font=\footnotesize] {$f_1$} (G);
        \path[->] (D)  edge node[fill=white, font=\footnotesize] {$s_1$} (Hp1);
        \path[->] (D1) edge node[fill=white, font=\footnotesize] {$g_1$} (H1);
        \path[->] (R1) edge node[fill=white, font=\footnotesize] {$i'_1$} (Rp1); 
        \path[->] (K1) edge node[fill=white, font=\footnotesize, near start] {$r'_1$} (Rp1); 
        \path[->] (Rp1) edge node[fill=white, font=\footnotesize] {$n_1$} (H1); 
        \path[->] (R1) edge [bend left = 40] node[fill=white, font=\footnotesize] {$o_1$} (Hp1); 
        \path[->] (Rp1) edge [dashed, bend left] node[fill=white, font=\footnotesize] {$j'_1$} (Hp1); 
        \path[->] (D)  edge node[fill=white, font=\footnotesize] {$e_1$} (D1);
        \path[->] (Hp1) edge node[fill=white, font=\footnotesize] {$h_1$} (H); 
        \path[->] (Hp1) edge [dashed] node[fill=white, font=\footnotesize] {$f'_2$} (H1); 
        \path[->] (L1) edge [bend right] node[fill=white,
        font=\footnotesize, near end] {$j_1$} (D2);
        \path[->] (K1) edge [bend left, dashed] node[fill=white,
        font=\footnotesize, near start] {$d'_1$} (D);
        \path[->] (I1) edge [bend left = 40, dashed] node[fill=white,
        font=\footnotesize] {$d_1$} (D);
        \path[->] (K2) edge node[fill=white, font=\footnotesize] {$l_2$} (L2);
        \path[->] (I2) edge node[fill=white, font=\footnotesize] {$i_2$} (K2);
        \path[->] (I2) edge node[fill=white, font=\footnotesize] {$r_2$} (R2);
        \path[->] (L2) edge node[fill=white, font=\footnotesize] {$m_2$} (G);
        \path[->] (K2) edge node[fill=white, font=\footnotesize] {$k_2$} (D2);
        \path[->] (D2) edge node[fill=white, font=\footnotesize] {$f_2$} (G);
        \path[->] (D)  edge node[fill=white, font=\footnotesize] {$s_2$} (Hp2);
        \path[->] (D2) edge node[fill=white, font=\footnotesize] {$g_2$} (H2);
        \path[->] (R2) edge node[fill=white, font=\footnotesize] {$n_2$} (H2); 
        \path[->] (R2) edge [bend right] node[fill=white, font=\footnotesize] {$o_2$} (Hp2); 
        \path[->] (D)  edge node[fill=white, font=\footnotesize] {$e_2$} (D2);
        \path[->] (Hp2) edge node[fill=white, font=\footnotesize] {$h_2$} (H);
        \path[-] (L2) edge[draw=white, line width=3pt, bend left]  (D1);
        \path[->] (L2) edge [bend left] node[fill=white,
        font=\footnotesize, near end] {$j_2$} (D1);
        \path[->] (K2) edge [bend right, dashed] node[fill=white,
        font=\footnotesize, near start] {$d'_2$} (D);
        \path[->] (I2) edge [bend right = 40, dashed] node[fill=white,
        font=\footnotesize] {$d_2$} (D);
    \end{tikzpicture}
    \caption{Analysis}\label{fig-analysis}
    \end{figure}
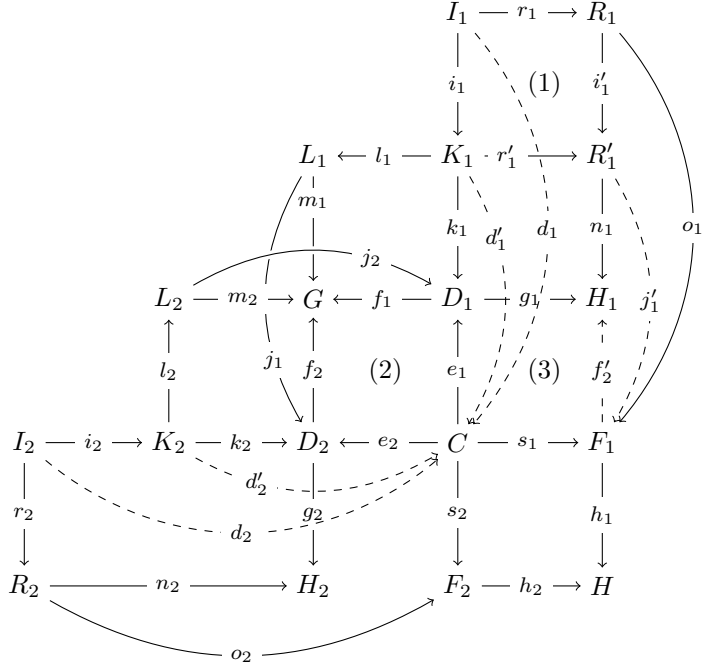
    Since (2) is a pullback and $f_2\circ j_1\circ l_1 = m_1\circ l_1 =
  f_1\circ k_1$, then there is a (unique) morphism $d'_1:K_1\rightarrow
  C$ such that $e_1\circ d'_1=k_1$ and $e_2\circ d'_1 = j_1\circ l_1$.
  By definition $d_1$ is the unique morphism such that $e_1\circ d_1=
  k_1\circ i_1$ and $e_2\circ d_1 = j^2_1 = j_1\circ l_1\circ i_1$. But
  $e_1\circ d'_1\circ i_1 = k_1\circ i_1$ and $e_2\circ
  d'_1\circ i_1 = j_1\circ l_1\circ i_1$, hence $d_1 =
  d'_1 \circ i_1$. Symmetrically we get a morphism $d'_2:K_2\rightarrow
  C$ such that $e_1\circ d'_2= j_2\circ l_2$, $e_2\circ d'_2 = k_2$
  and $d_2 = d'_2 \circ i_2$.

  By definition of $G\trans{\tuple{\adt_1,\adt_2}} H$ the square
  $(1)+(5)$ below is a pushout, hence $s_1\circ d'_1\circ i_1 =
  o_1\circ r_1$, and since (1) is also a pushout then there exists
  $j'_1: R'_1\rightarrow F_1$ such that $s_1\circ d'_1 = j'_1\circ
  r'_1$ and $o_1=j'_1\circ i'_1$, hence by decomposition (5) is also a pushout.
  \begin{center}
        \begin{tikzpicture}[yscale=1.5,xscale=2]
      \node (R1) at (0,0) {$R_1$};
      \node (I1) at (0,1) {$I_1$};
      \node (K1) at (1,1) {$K_1$};
      \node (Rp1) at (1,0) {$R'_1$};
      \node (Dp) at (2,1) {$C$}; 
      \node (Hp1) at (2,0) {$F_1$}; 
      \node (D1) at (3,1) {$D_1$}; 
      \node (H1) at (3,0) {$H_1$}; \node at (0.5,0.5) {(1)}; 
      \node at (1.5,0.5) {(5)}; \node at (2.5,0.5) {(3)}; 
      \path[->] (I1) edge node[fill=white, font=\footnotesize] {$i_1$} (K1);
      \path[->] (I1) edge node[fill=white, font=\footnotesize] {$r_1$} (R1);
      \path[->] (K1) edge node[fill=white, font=\footnotesize] {$r'_1$} (Rp1);
      \path[->] (R1) edge node[fill=white, font=\footnotesize] {$i'_1$} (Rp1);
      \path[->] (K1) edge node[fill=white, font=\footnotesize] {$d'_1$} (Dp);
      \path[->] (Dp) edge node[fill=white, font=\footnotesize] {$s_1$} (Hp1);
      \path[->, dashed] (Rp1) edge node[fill=white, font=\footnotesize] {$j'_1$} (Hp1);
      \path[->] (I1) edge [bend left] node[fill=white, font=\footnotesize] {$d_1$} (Dp);
      \path[->] (R1) edge [bend right] node[fill=white, font=\footnotesize] {$o_1$} (Hp1);
      \path[->] (Dp) edge node[fill=white, font=\footnotesize] {$e_1$} (D1);
      \path[->] (D1) edge node[fill=white, font=\footnotesize] {$g_1$} (H1);
      \path[->, dashed] (Hp1) edge node[fill=white, font=\footnotesize] {$f'_2$} (H1);
      \path[-] (K1) edge [bend left, draw=white, line width=3pt] (D1);
      \path[-] (Rp1) edge [bend right, draw=white, line width=3pt] (H1);
      \path[->] (K1) edge [bend left] node[fill=white, font=\footnotesize] {$k_1$} (D1);
      \path[->] (Rp1) edge [bend right] node[fill=white, font=\footnotesize] {$n_1$} (H1);
    \end{tikzpicture}
  \end{center}
  Similarly the square (5)+(3) is a pushout, hence $g_1\circ e_1\circ
  d'_1= n_1\circ r'_1$ and there exists $f'_2:F_1\rightarrow H_1$
  such that $g_1\circ e_1=f'_2\circ s_1$ and $n_1 = f'_2\circ j'_1$,
  and by decomposition (3) is a pushout. In the diagram
  \begin{center}
    \begin{tikzpicture}[yscale=1.5,xscale=2]
      \node (G) at (0,0) {$L_2$};
      \node (L) at (0,1) {$K_2$};
      \node (K) at (1,1) {$C$};
      \node (D) at (1,0) {$D_1$};
      \node (RK) at (2,1) {$D_2$};
      \node (H) at (2,0) {$G$};
      \node at (0.5,0.5) {(6)}; \node at (1.5,0.5) {(2)};
      \path[<-] (K) edge node[fill=white, font=\footnotesize] {$d'_2$} (L);
      \path[->] (L) edge node[fill=white, font=\footnotesize] {$l_2$} (G);
      \path[->] (K) edge node[fill=white, font=\footnotesize] {$e_1$} (D);
      \path[<-] (D) edge node[fill=white, font=\footnotesize] {$j_2$} (G);
      \path[->] (D) edge node[fill=white, font=\footnotesize] {$f_1$} (H);
      \path[->] (K) edge node[fill=white, font=\footnotesize] {$e_2$} (RK);
      \path[->] (RK) edge node[fill=white, font=\footnotesize] {$f_2$} (H);
      \path[->] (L) edge [bend left] node[fill=white, font=\footnotesize] {$k_2$} (RK);
      \path[->] (G) edge [bend right] node[fill=white, font=\footnotesize] {$m_2$} (H);
    \end{tikzpicture}
  \end{center}
  the square $(6)+(2)$ is a pushout and (2) is a pullback, hence
  by the $\M$-POPB decomposition lemma, (6) is a
  pushout. By composition we deduce that (6)+(3) is a pushout, and
  similarly for the right pushout of the following
  $\adt'_2\in \dirtranspo{H_1}{\arule_2}$ where $m'_2 = g_1\circ j_2$,
  $k'_2 = s_1\circ d'_2$ and $n'_2=h_2\circ o_2$. 
  \begin{center}
  \begin{tikzpicture}[xscale=1.65, yscale=1.8]
    \node (L) at (0.5,1) {$L_2$}; \node (K) at (1.5,1) {$K_2$};  \node (I) at
    (2.5,1) {$I_2$};  \node (R) at (3.5,1) {$R_2$}; \node (G) at (0,0) {$H_1$};
    \node (D) at (2,0) {$F_1$}; \node (H) at (3,0) {$H$};
    \path[->] (K) edge node[fill=white, font=\footnotesize] {$l_2$} (L);
    \path[->] (I) edge node[fill=white, font=\footnotesize] {$i_2$} (K);
    \path[->] (I) edge node[fill=white, font=\footnotesize] {$r_2$} (R);
    \path[->] (L) edge node[right, font=\footnotesize, near end] {${m'_2}$} (G);
    \path[->] (K) edge node[fill=white, font=\footnotesize] {$k'_2$} (D);
    \path[->] (D) edge node[fill=white, font=\footnotesize] {$f'_2$} (G);
    \path[->] (I) edge node[fill=white, font=\footnotesize] {$k'_2\circ i_2$} (D);
    \path[->] (D) edge node[fill=white, font=\footnotesize] {$h_1$} (H);
    \path[->] (R) edge node[fill=white, font=\footnotesize] {$n'_2$} (H);
    \node (L1) at (-0.5,1) {$R'_1$}; \node (K1) at (-1.5,1) {$K_1$};
    \node (D1) at (-1.5,0) {$D_1$}; 
    \path[->] (K1) edge node[fill=white, font=\footnotesize] {$r'_1$} (L1);
    \path[->] (L1) edge node[left, font=\footnotesize, near end] {$n_1$} (G);
    \path[->] (K1) edge node[fill=white, font=\footnotesize] {$k_1$} (D1);
    \path[->] (D1) edge node[fill=white, font=\footnotesize] {$g_1$} (G);
  \path[-] (L) edge[draw=white, line width=3pt]  (D1);
    \path[->] (L) edge node[fill=white, font=\footnotesize] {$j_2$} (D1); 
  \path[-] (L1) edge[draw=white, line width=3pt]  (D); 
    \path[->] (L1) edge node[fill=white, font=\footnotesize] {$j'_1$} (D);
  \end{tikzpicture}    
  \end{center}
  Here, $\adt_1$ is only partially depicted. We have $f'_2\circ j'_1 =
  n_1$, hence $G\trans{\adt_1} H_1\trans{\adt'_2} H$ is sequential
  independent. 
\item We consider a direct transformation $G\trans{\adt'_2} H$ as
  above, with morphisms $j'_1: R'_1\rightarrow F_1$ and
  $j_2: L_2\rightarrow D_1$ such that $f'_2\circ j'_1=n_1$ and
  $g_1\circ j_2=m'_2$. Since $l_2\in\M$ then $f'_2\in\M$ hence there
  exists $\tuple{C,e_1,s_1}$ such that the square (3) in
  Figure~\ref{fig-synthesis}  is a pullback, so that $e_1\in\M$ as
  well. 
  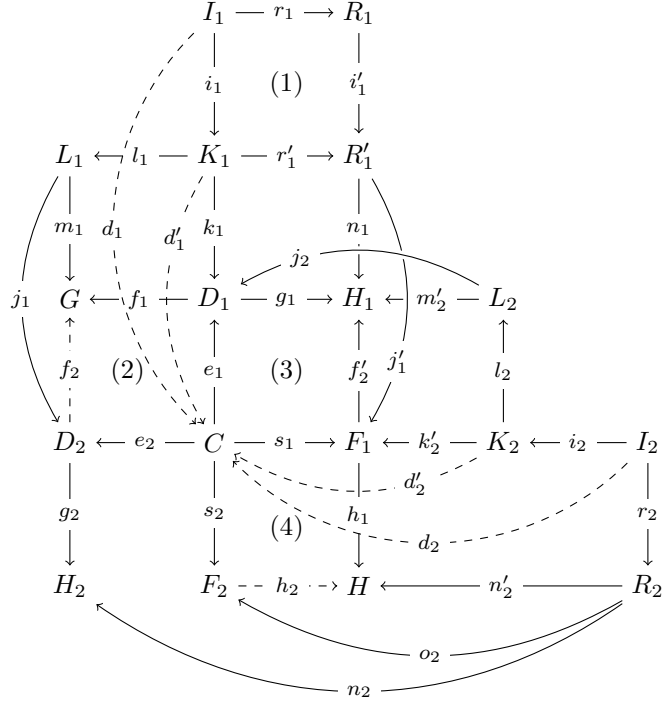
\begin{figure}[t]
    \centering
    \begin{tikzpicture}[scale=1.9]
        \node (G) at (0,0) {$G$};
        \node (D) at (1,-1) {$C$};
        \node (H) at (2,-2) {$H$};
        \node (L1) at (0,1) {$L_1$};
        \node (K1) at (1,1) {$K_1$};
        \node (D1) at (1,0) {$D_1$};
        \node (H1) at (2,0) {$H_1$};
        \node (I1) at (1,2) {$I_1$};
        \node (R1) at (2,2) {$R_1$}; 
        \node (Rp1) at (2,1) {$R'_1$}; 
        \node (Hp1) at (2,-1) {$F_1$}; 
        \node (L2) at (3,0) {$L_2$};
        \node (K2) at (3,-1) {$K_2$};
        \node (D2) at (0,-1) {$D_2$};
        \node (H2) at (0,-2) {$H_2$};
        \node (I2) at (4,-1) {$I_2$};
        \node (R2) at (4,-2) {$R_2$};
        \node (Hp2) at (1,-2) {$F_2$};
        \node at (1.5,1.5) {(1)}; \node at (0.4,-0.5) {(2)};
        \node at (1.5,-0.5) {(3)}; \node at (1.5,-1.6) {(4)}; 
        \path[->] (K1) edge node[fill=white, font=\footnotesize] {$l_1$} (L1);
        \path[->] (I1) edge node[fill=white, font=\footnotesize] {$i_1$} (K1);
        \path[->] (I1) edge node[fill=white, font=\footnotesize] {$r_1$} (R1);
        \path[->] (L1) edge node[fill=white, font=\footnotesize] {$m_1$} (G);
        \path[->] (K1) edge node[fill=white, font=\footnotesize] {$k_1$} (D1);
        \path[->] (D1) edge node[fill=white, font=\footnotesize] {$f_1$} (G);
        \path[->] (D)  edge node[fill=white, font=\footnotesize] {$s_1$} (Hp1);
        \path[->] (D1) edge node[fill=white, font=\footnotesize] {$g_1$} (H1);
        \path[->] (R1) edge node[fill=white, font=\footnotesize] {$i'_1$} (Rp1); 
        \path[->] (K1) edge node[fill=white, font=\footnotesize] {$r'_1$} (Rp1); 
        \path[->] (Rp1) edge node[fill=white, font=\footnotesize] {$n_1$} (H1); 
        \path[->] (D)  edge node[fill=white, font=\footnotesize] {$e_1$} (D1);
        \path[->] (Hp1) edge node[fill=white, font=\footnotesize] {$h_1$} (H); 
        \path[->] (Hp1) edge node[fill=white, font=\footnotesize] {$f'_2$} (H1); 
        \path[->] (R2) edge node[fill=white, font=\footnotesize] {$n'_2$} (H); 
        \path[->] (L1) edge [bend right] node[fill=white, font=\footnotesize] {$j_1$} (D2);
        \path[->] (K1) edge [bend right, dashed] node[fill=white,
        font=\footnotesize, near start] {$d'_1$} (D);
        \path[->] (I1) edge [bend right = 45, dashed] node[fill=white, font=\footnotesize] {$d_1$} (D);
        \path[->] (K2) edge node[fill=white, font=\footnotesize] {$l_2$} (L2);
        \path[->] (I2) edge node[fill=white, font=\footnotesize] {$i_2$} (K2);
        \path[->] (I2) edge node[fill=white, font=\footnotesize] {$r_2$} (R2);
        \path[->] (L2) edge node[fill=white, font=\footnotesize] {$m'_2$} (H1);
        \path[-] (Rp1) edge [bend left, draw=white, line width=3pt] (Hp1); 
        \path[->] (Rp1) edge [bend left] node[fill=white, font=\footnotesize, near end] {$j'_1$} (Hp1); 
        \path[->] (K2) edge node[fill=white, font=\footnotesize] {$k'_2$} (Hp1);
        \path[->,dashed] (D2) edge node[fill=white, font=\footnotesize] {$f_2$} (G);
        \path[->] (D)  edge node[fill=white, font=\footnotesize] {$s_2$} (Hp2);
        \path[->] (D2) edge node[fill=white, font=\footnotesize] {$g_2$} (H2);
        \path[->,bend left=35] (R2) edge node[fill=white, font=\footnotesize] {$n_2$} (H2); 
        \path[->] (R2) edge [bend left] node[fill=white, font=\footnotesize] {$o_2$} (Hp2); 
        \path[->] (D)  edge node[fill=white, font=\footnotesize] {$e_2$} (D2);
        \path[->,dashed] (Hp2) edge node[fill=white, font=\footnotesize] {$h_2$} (H);
        \path[-] (L2) edge[draw=white, line width=3pt, bend right]  (D1);
        \path[->] (L2) edge [bend right] node[fill=white, font=\footnotesize, near end] {$j_2$} (D1); 
        \path[->] (K2) edge [bend left, dashed] node[fill=white,
        font=\footnotesize, near start] {$d'_2$} (D);
        \path[->] (I2) edge [bend left = 45, dashed] node[fill=white,font=\footnotesize] {$d_2$} (D);
    \end{tikzpicture}
    \caption{Synthesis}\label{fig-synthesis}
  \end{figure}

  We have $g_1\circ k_1 = n_1\circ r'_1 = f'_2\circ j'_1\circ r'_1$,
  hence there exists a unique $d'_1:K_1\rightarrow C$ such that
  $e_1\circ d'_1 = k_1$ and $s_1\circ d'_1 = j'_1\circ
  r'_1$. Similarly, we have $f'_2\circ k'_2 = m'_2\circ l_2 = g_1\circ
  j_2\circ l_2$, hence there exists a unique $d'_2:K_2\rightarrow C$ such that
  $e_1\circ d'_2 = j_2\circ l_2$ and $s_1\circ d'_2 = k'_2$.

  Since $l_1\in\M$ then there exists $\tuple{j_1,e_2,D_2}$ such that
  the square (7) in the diagram below is a pushout.
  \begin{center}
    \begin{tikzpicture}[yscale=1.5,xscale=2]
      \node (G) at (0,0) {$L_1$};
      \node (L) at (0,1) {$K_1$};
      \node (K) at (1,1) {$C$};
      \node (D) at (1,0) {$D_2$};
      \node (RK) at (2,1) {$D_1$};
      \node (H) at (2,0) {$G$};
      \node at (0.5,0.5) {(7)}; \node at (1.5,0.5) {(2)};
      \path[<-] (K) edge node[fill=white, font=\footnotesize] {$d'_1$} (L);
      \path[->] (L) edge node[fill=white, font=\footnotesize] {$l_1$} (G);
      \path[->] (K) edge node[fill=white, font=\footnotesize] {$e_2$} (D);
      \path[<-] (D) edge node[fill=white, font=\footnotesize] {$j_1$} (G);
      \path[->,dashed] (D) edge node[fill=white, font=\footnotesize] {$f_2$} (H);
      \path[->] (K) edge node[fill=white, font=\footnotesize] {$e_1$} (RK);
      \path[->] (RK) edge node[fill=white, font=\footnotesize] {$f_1$} (H);
      \path[->] (L) edge [bend left] node[fill=white, font=\footnotesize] {$k_1$} (RK);
      \path[->] (G) edge [bend right] node[fill=white, font=\footnotesize] {$m_1$} (H);
    \end{tikzpicture}
  \end{center}
  We have $m_1\circ l_1 = f_1\circ k_1 = f_1\circ e_1\circ d'_1$,
  hence there exists a unique $f_2:D_2\rightarrow G$ such that
  $f_2\circ j_1=m_1$ and $f_2\circ e_2 = f_1\circ e_1$. Since
  $(7)+(2)$ is a pushout then by decomposition (2) is a pushout, and
  since $e_1\in\M$, then it is also a pullback.

  In the diagram below
  
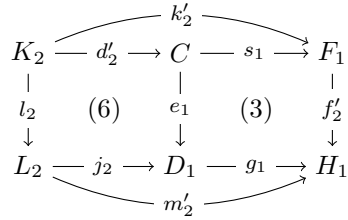
\begin{center}
    \begin{tikzpicture}[yscale=1.5,xscale=2]
      \node (G) at (0,0) {$L_2$};
      \node (L) at (0,1) {$K_2$};
      \node (K) at (1,1) {$C$};
      \node (D) at (1,0) {$D_1$};
      \node (RK) at (2,1) {$F_1$};
      \node (H) at (2,0) {$H_1$};
      \node at (0.5,0.5) {(6)}; \node at (1.5,0.5) {(3)};
      \path[<-] (K) edge node[fill=white, font=\footnotesize] {$d'_2$} (L);
      \path[->] (L) edge node[fill=white, font=\footnotesize] {$l_2$} (G);
      \path[->] (K) edge node[fill=white, font=\footnotesize] {$e_1$} (D);
      \path[<-] (D) edge node[fill=white, font=\footnotesize] {$j_2$} (G);
      \path[->] (D) edge node[fill=white, font=\footnotesize] {$g_1$} (H);
      \path[->] (K) edge node[fill=white, font=\footnotesize] {$s_1$} (RK);
      \path[->] (RK) edge node[fill=white, font=\footnotesize] {$f'_2$} (H);
      \path[->] (L) edge [bend left] node[fill=white, font=\footnotesize] {$k'_2$} (RK);
      \path[->] (G) edge [bend right] node[fill=white, font=\footnotesize] {$m'_2$} (H);
    \end{tikzpicture}
  \end{center}
  we know that (3) is a pullback and that $(6)+(3)$ is a
  pushout. Besides, $r_1\in\M$ and (1) is a pushout, hence $r'_1\in\M$
  and similarly $g_1\in\M$, hence by the $\M$-POPB decomposition lemma
  (6) and (3) are pushouts. By composition we have that $(6)+(2)$ is a
  pushout, and since $r_2\in \M$ then there exists a
  pushout $\tuple{n_2,g_2, H_2}$ of $\tuple{I_2,e_2\circ d'_2\circ
    i_2, r_2}$. This yields the direct transformation
  $\adt_2\in\dirtranspo{G}{\arule_2}$ depicted below,
  
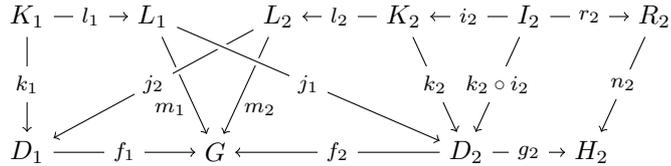
\begin{center}
  \begin{tikzpicture}[xscale=1.65, yscale=1.8]
    \node (L) at (0.5,1) {$L_2$}; \node (K) at (1.5,1) {$K_2$};  \node (I) at
    (2.5,1) {$I_2$};  \node (R) at (3.5,1) {$R_2$}; \node (G) at (0,0) {$G$};
    \node (D) at (2,0) {$D_2$}; \node (H) at (3,0) {$H_2$};
    \path[->] (K) edge node[fill=white, font=\footnotesize] {$l_2$} (L);
    \path[->] (I) edge node[fill=white, font=\footnotesize] {$i_2$} (K);
    \path[->] (I) edge node[fill=white, font=\footnotesize] {$r_2$} (R);
    \path[->] (L) edge node[right, font=\footnotesize, near end] {${m_2}$} (G);
    \path[->] (K) edge node[fill=white, font=\footnotesize] {$k_2$} (D);
    \path[->] (D) edge node[fill=white, font=\footnotesize] {$f_2$} (G);
    \path[->] (I) edge node[fill=white, font=\footnotesize] {$k_2\circ i_2$} (D);
    \path[->] (D) edge node[fill=white, font=\footnotesize] {$g_2$} (H);
    \path[->] (R) edge node[fill=white, font=\footnotesize] {$n_2$} (H);
    \node (L1) at (-0.5,1) {$L_1$}; \node (K1) at (-1.5,1) {$K_1$};
    \node (D1) at (-1.5,0) {$D_1$}; 
    \path[->] (K1) edge node[fill=white, font=\footnotesize] {$l_1$} (L1);
    \path[->] (L1) edge node[left, font=\footnotesize, near end] {$m_1$} (G);
    \path[->] (K1) edge node[fill=white, font=\footnotesize] {$k_1$} (D1);
    \path[->] (D1) edge node[fill=white, font=\footnotesize] {$f_1$} (G);
  \path[-] (L) edge[draw=white, line width=3pt]  (D1);
    \path[->] (L) edge node[fill=white, font=\footnotesize] {$j_2$} (D1); 
  \path[-] (L1) edge[draw=white, line width=3pt]  (D); 
    \path[->] (L1) edge node[fill=white, font=\footnotesize] {$j_1$} (D);
  \end{tikzpicture}    
  \end{center}
  where $k_2= e_2\circ d'_2$ and $m_2= f_1\circ j_2$. By definition of
  $f_2$ we have $f_2\circ j_1=m_1$, hence $\adt_1$ and $\adt_2$ are
  parallel independent.

  Let $j_1^2=j_1\circ l_1\circ i_1$, we have $f_2\circ j_1^2 =
  m_1\circ l_1\circ i_1 = f_1\circ k_1\circ i_1$, and since (2) is a
  pullback there exists a unique $d_1:I_1\rightarrow C$ such that
  $e_1\circ d_1 = k_1\circ i_1$ and $e_2\circ d_1 = j_1^2$, and as
  above we easily see that $d_1 = d'_1\circ i_1$. Let $o_1=j'_1\circ
  i'_1$, in the diagram below
  \begin{center}
        \begin{tikzpicture}[yscale=1.5,xscale=2]
      \node (R1) at (0,0) {$R_1$};
      \node (I1) at (0,1) {$I_1$};
      \node (K1) at (1,1) {$K_1$};
      \node (Rp1) at (1,0) {$R'_1$};
      \node (Dp) at (2,1) {$C$}; 
      \node (Hp1) at (2,0) {$F_1$}; 
      \node (D1) at (3,1) {$D_1$}; 
      \node (H1) at (3,0) {$H_1$}; \node at (0.5,0.5) {(1)}; 
      \node at (1.5,0.5) {(5)}; \node at (2.5,0.5) {(3)}; 
      \path[->] (I1) edge node[fill=white, font=\footnotesize] {$i_1$} (K1);
      \path[->] (I1) edge node[fill=white, font=\footnotesize] {$r_1$} (R1);
      \path[->] (K1) edge node[fill=white, font=\footnotesize] {$r'_1$} (Rp1);
      \path[->] (R1) edge node[fill=white, font=\footnotesize] {$i'_1$} (Rp1);
      \path[->] (K1) edge node[fill=white, font=\footnotesize] {$d'_1$} (Dp);
      \path[->] (Dp) edge node[fill=white, font=\footnotesize] {$s_1$} (Hp1);
      \path[->] (Rp1) edge node[fill=white, font=\footnotesize] {$j'_1$} (Hp1);
      \path[->] (I1) edge [bend left] node[fill=white, font=\footnotesize] {$d_1$} (Dp);
      \path[->] (R1) edge [bend right] node[fill=white, font=\footnotesize] {$o_1$} (Hp1);
      \path[->] (Dp) edge node[fill=white, font=\footnotesize] {$e_1$} (D1);
      \path[->] (D1) edge node[fill=white, font=\footnotesize] {$g_1$} (H1);
      \path[->] (Hp1) edge node[fill=white, font=\footnotesize] {$f'_2$} (H1);
      \path[-] (K1) edge [bend left, draw=white, line width=3pt] (D1);
      \path[-] (Rp1) edge [bend right, draw=white, line width=3pt] (H1);
      \path[->] (K1) edge [bend left] node[fill=white, font=\footnotesize] {$k_1$} (D1);
      \path[->] (Rp1) edge [bend right] node[fill=white, font=\footnotesize] {$n_1$} (H1);
    \end{tikzpicture}
  \end{center}
  we know that (3) is a pullback, $(5)+(3)$ is a pushout and $r'_1,
  f'_2\in \M$, hence by the $\M$-POPB decomposition lemma (5) is a
  pushout, but so is (1), hence by composition $(1)+(5)$ is a pushout.

  Let $j_2^1 = j_2\circ l_2\circ i_2 = e_1\circ d'_2\circ i_2$ (by
  definition of $d'_2$), we have $f_1\circ j_2^1 = f_2\circ e_2\circ
  d'_2\circ i_2 = f_2\circ k_2\circ i_2$, and since (2) is a
  pullback there exists a unique $d_2:I_2\rightarrow C$ such that
  $e_1\circ d_2 = j_2^1$ and $e_2\circ d_2 = k_2\circ i_2$, and as
  above we easily see that $d_2 = d'_2\circ i_2$.

  Since $r_2\in\M$ then there exists $\tuple{F_2,s_2,o_2}$ such that
  the square (8) below is a pushout.
  \begin{center}
    \begin{tikzpicture}[yscale=1.5,xscale=2]
      \node (G) at (0,0) {$R_2$};
      \node (L) at (0,1) {$I_2$};
      \node (K) at (1,1) {$C$};
      \node (D) at (1,0) {$F_2$};
      \node (RK) at (2,1) {$F_1$};
      \node (H) at (2,0) {$H$};
      \node at (0.5,0.5) {(8)}; \node at (1.5,0.5) {(4)};
      \path[<-] (K) edge node[fill=white, font=\footnotesize] {$d_2$} (L);
      \path[->] (L) edge node[fill=white, font=\footnotesize] {$r_2$} (G);
      \path[->] (K) edge node[fill=white, font=\footnotesize] {$s_2$} (D);
      \path[<-] (D) edge node[fill=white, font=\footnotesize] {$o_2$} (G);
      \path[->,dashed] (D) edge node[fill=white, font=\footnotesize] {$h_2$} (H);
      \path[->] (K) edge node[fill=white, font=\footnotesize] {$s_1$} (RK);
      \path[->] (RK) edge node[fill=white, font=\footnotesize] {$h_1$} (H);
      \path[->] (L) edge [bend left] node[fill=white,
      font=\footnotesize] {$k'_2\circ i_2$} (RK);
      \path[->] (G) edge [bend right] node[fill=white, font=\footnotesize] {$n'_2$} (H);
    \end{tikzpicture}
  \end{center}
  We have
  $h_1\circ s_1\circ d_2 = h_1\circ k'_2\circ i_2 = n'_2\circ r_2$ (by
  definition of $d'_2$ and $\adt'_2$), hence there exists a unique
  $h_2:F_2\rightarrow H$ such that $h_2\circ o_2=n'_2$ and
  $h_2\circ s_2 = h_1\circ s_1$. But $(8)+(4)$ is a pushout, hence by
  decomposition (4) is a pushout, which yields a PCT
  $G\trans{\tuple{\adt_1,\adt_2}} H$.

\item Since pushouts and pullbacks are unique up to isomorphism, the
  PCT $G\trans{\tuple{\adt_1,\adt_2}} H$ is determined (up to
  isomorphism) by $\adt_1$, $\adt_2$, $j_1^2$ and $j_2^1$. But we have
  seen that $j_1^2=j_1\circ l_1\circ i_1$ and
  $j_2^1 = j_2\circ l_2\circ i_2$, i.e., they are determined by $j_1$
  and $j_2$ ($\arule_1$ and $\arule_2$ being invariant in the
  considered correspondences). But it is easy to see that these are
  determined by $\adt_1$ and $\adt_2$, for instance if there is a
  $j'_2$ such that $g_1\circ j'_2 = n_2 = g_1\circ j_2$, since
  $g_1\in\M$ is a monomorphism then $j'_2=j_2$. As $\adt_1$ is also
  invariant, the PCT is determined by $\adt_2$. Since
  $e_1\circ d'_2 = j_2\circ l_2$ and $e_1\in\M$, then $d'_2$ is
  determined as well by $\adt_2$.
    
    For the same reasons $\adt_2$ (resp. $\adt'_2$) is determined by
    $k_2$ (resp. $k'_2$). In the first correspondence above we
    associate to $k_2=e_2\circ d'_2$ the morphism
    $k'_2=s_1\circ d'_2$, while in the second one we associate to
    $k'_2=s_1\circ d'_2$ the morphism $k_2=e_2\circ d'_2$. We have
    seen that $e_2$, $s_1$ and $d'_2$ are all determined by $k_2$,
    hence the two correspondences are inverse to each other.
  \end{enumerate} 
\end{proof}

This means that a PCT of $G$ by two parallel independent direct
transformations yields a result that can be obtained by a sequence of
two direct transformations, in any order (they are sequential
independent). This can be interpreted as a result of correctness of
PCTs w.r.t. the standard approach to
(independent) parallelism of algebraic graph transformations. In this
sense, parallel coherence is a conservative extension of parallel
independence.

As mentioned in the abstract, contrary to the standard Parallelism
Theorem \cite{EhrigEPT06} Theorem~\ref{th-parallelism} does not need
to assume that $\aCat$ has coproducts compatible with $\M$, i.e., such
that $f+g\in\M$ whenever $f,g\in\M$.

\section{Derived Rules}

\newcommand{\relobj}{\vdash}

\begin{definition}
  Given a binary relation $\relobj$ beween objects of $\aCat$ and $G$,
  $H$ such that $G\relobj H$, a \emph{derived rule of $G\relobj H$} is
  a span $\aspan$ in $\aCat$ such that $G \trans{\aspan} H$ and for
  all objects $G',H'$ of $\aCat$, if $G' \trans{\aspan} H'$ then
  $G'\relobj H'$.

  For any objects $G$, $H$ and weak-span $\arule$, we write
  $G\ruledir{\arule}H$  if there exists a diagram
  $\adt\in\dirtranspo{G}{\arule}$ such that $G\trans{\adt}H$. We write
  $G\ruletrans{\R} H$ if there
  exists a parallel coherent diagram $\Gamma$ where
  $\adt_1,\dotsc,\adt_p\in \dirtranspo{G}{\R}$ such that
  $G\trans{\Gamma} H$.
\end{definition}

By Lemma \ref{lm-assocspan}, for every $\M$-weak span $\arule$ and every
objects $G,H$ such that $G\ruledir{\arule} H$, the associated span
$\assoc{\arule}$ is a derived rule of $G\ruledir{\arule} H$. The span
$\assoc{\arule}$ does not depend on $G$ or $H$, but this is not
generally true of derived rules.

\begin{lemma}\label{lm-indstep-limit}
  For any $p\geq 2$, any sink $(f_1,\dotsc,f_p,G)$ of domain
  $(D_a)_{a=1}^p$, any limit $(B,g_2,\dotsc,g_p)$ of
  $(f_2,\dotsc,f_p,G)$ and any source $(C,e_1,\dotsc,e_p)$ of codomain
  $(D_a)_{a=1}^p$, then $(C,e_1,\dotsc,e_p)$ is a limit of
  $(f_1,\dotsc,f_p,G)$ iff there exists a morphism $x:C\rightarrow B$
  such that $g_a\circ x=e_a$ for all $2\leq a\leq p$ and the diagram
  \begin{center}
    \begin{tikzpicture}[xscale=-2.5,yscale=1.5]
      \node (B) at (0,0) {$B$}; \node (C) at (0,1){$C$}; \node (G) at
      (1,0) {$G$}; \node (D1) at (1,1) {$D_1$};
      \path[->] (C) edge node[fill=white, font=\footnotesize] {$x$}(B); 
      \path[->] (C) edge node[fill=white, font=\footnotesize] {$e_1$}(D1); 
      \path[->] (D1) edge node[fill=white, font=\footnotesize] {$f_1$}(G); 
      \path[->] (B) edge node[fill=white, font=\footnotesize] {$f_2\circ g_2$}(G); 
    \end{tikzpicture}
  \end{center}
  is a pullback.
\end{lemma}
\begin{proof}
  If part. We assume that $(C,e_1,x)$ is a pullback of
  $(f_1,f_2\circ g_2, G)$,
  \begin{center}
    \begin{tikzpicture}[xscale=1.5,yscale=0.8]
      \node (G) at (0,0) {$G$}; \node (D1) at (1,1) {$D_1$};
      \node (D2) at (1,0) {$D_2$}; \node at (1,-0.9){$\vdots$};
      \node (Dp) at (1,-2){$D_p$}; \node (B) at (2,-1){$B$};
      \node (C) at (3,0){$C$};
      \path[->] (D1) edge node [fill=white, font=\footnotesize]{$f_1$} (G);
      \path[->] (D2) edge node [fill=white, font=\footnotesize]{$f_2$} (G);
      \path[->] (Dp) edge node [fill=white, font=\footnotesize]{$f_p$} (G);
      \path[->] (B) edge node [fill=white, font=\footnotesize]{$g_2$} (D2);
      \path[->] (B) edge node [fill=white, font=\footnotesize]{$g_p$} (Dp);
      \path[->] (C) edge node [fill=white, font=\footnotesize]{$e_1$} (D1);
      \path[->] (C) edge node [fill=white, font=\footnotesize]{$x$} (B);
    \end{tikzpicture}
  \end{center}
  it is then easy to prove that $(C,e_1,g_2\circ x,\dotsc, g_p\circ
  x)$ is a limit of $(f_1,\dotsc, f_p,G)$.

  Only if part. We assume that $(C,e_1,\dotsc,e_p)$ is a limit of
  $(f_1,\dotsc,f_p,G)$,
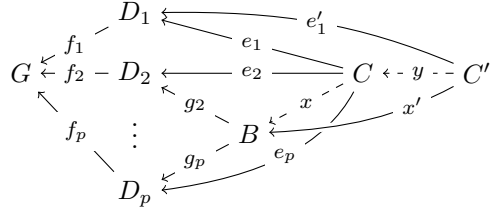
  \begin{center}
    \begin{tikzpicture}[xscale=1.5,yscale=0.8]
      \node (G) at (0,0) {$G$}; \node (D1) at (1,1) {$D_1$};
      \node (D2) at (1,0) {$D_2$}; \node at (1,-0.9){$\vdots$};
      \node (Dp) at (1,-2){$D_p$}; \node (B) at (2,-1){$B$};
      \node (C) at (3,0){$C$};\node (C') at (4,0) {$C'$};
      \path[->] (D1) edge node [fill=white, font=\footnotesize]{$f_1$} (G);
      \path[->] (D2) edge node [fill=white, font=\footnotesize]{$f_2$} (G);
      \path[->] (Dp) edge node [fill=white, font=\footnotesize]{$f_p$} (G);
      \path[->] (B) edge node [fill=white, font=\footnotesize]{$g_2$} (D2);
      \path[->] (B) edge node [fill=white, font=\footnotesize]{$g_p$} (Dp);
      \path[->] (C) edge node [fill=white, font=\footnotesize]{$e_1$} (D1);
      \path[->] (C) edge node [fill=white, font=\footnotesize]{$e_2$} (D2);
      \path[->] (C) edge [bend left] node [fill=white, font=\footnotesize]{$e_p$} (Dp);
      \path[->, dashed] (C) edge node [fill=white, font=\footnotesize]{$x$} (B);
      \path[-] (C') edge [bend left=20, draw=white, line width = 3pt] (B);
      \path[->] (C') edge [bend left=20] node [near start,fill=white, font=\footnotesize]{$x'$} (B);
      \path[->] (C') edge [bend right=20] node [fill=white, font=\footnotesize]{$e'_1$} (D1);
      \path[->, dashed] (C') edge node [fill=white, font=\footnotesize]{$y$} (C);
    \end{tikzpicture}
  \end{center}
  then $f_2\circ e_2=\dotsb=f_p\circ e_p$, and since
  $(B,g_2,\dotsc,g_p)$ is a limit then $\exists !x:C\rightarrow B$
  such that $e_a=g_a\circ x$ for all $2\leq a\leq p$, and the $g_a$'s
  are monomorphisms. Let $(C',h'1,x')$ such that
  $f_1\circ e'_1=f_2\circ g_2\circ x'$, since
  $f_2\circ g_2=\dotsb=f_p\circ g_p$ then by the hypothesis
  $\exists !y:C'\rightarrow C$ such that $e'_1=e_1\circ y$ and
  $g_a\circ x'=e_a\circ y = g_a\circ x\circ y$ for all
  $2\leq a\leq p$. But $(B,g_2,\dotsc,g_p)$ is a mono-source hence
  $x'=x\circ y$, which proves
  that $(C,e_1,x)$ is a pullback of $(f_1,f_2\circ g_2,G)$. 
\end{proof}

\begin{lemma}\label{lm-indstep-Mlimit}
  For all $p\geq 1$ and sink $(f_1,\dotsc,f_p,G)$ with
  $f_1,\dotsc,f_p\in\M$, there exists a limit $(C,e_1,\dotsc,e_p)$ of
  $(f_1,\dotsc,f_p,G)$ with $e_1,\dotsc,e_p\in\M$.
\end{lemma}
\begin{proof}
  By induction on $p$. If $p=1$ then $(D,\id{D})$ is a limit of
  $f_1:D\rightarrow G$, and $\id{D}\in\M$.

  Assume it is true for $p-1\leq 1$, then there exists a limit
  $(B,g_2,\dotsc,g_p)$ of $(f_2,\dotsc,f_p,G)$ with
  $g_2,\dotsc,g_p\in\M$. Since $f_1\in\M$ then there exists a pullback
  $(C,x,e_1)$ of $(f_1,f_2\circ g_2,G)$ and $x\in\M$, so that
  $g_a\circ x\in\M$ for all $2\leq a\leq p$. Similarly
  $f_2\circ g_2\in\M$ hence $e_1\in\M$, and by
  Lemma~\ref{lm-indstep-limit} $(C,e_1, g_2\circ x,\dotsc,g_p\circ x)$
  is a limit of $(f_1,\dotsc,f_p,G)$. 
\end{proof}

\begin{lemma}\label{lm-shift-limit}
  For all $p\geq 1$, sink $(f_1,\dotsc,f_p,G)$ with
  $f_1,\dotsc,f_p\in\M$, limit $(C,e_1,\dotsc,e_p)$ of
  $(f_1,\dotsc,f_p,G)$ with $e_1,\dotsc,e_p\in\M$, morphisms $g:G\rightarrow G'$, $c:C\rightarrow
  C'$ and $t_a,f'_a,e'_a$ below,
  \begin{center}
    \begin{tikzpicture}[xscale=2,yscale=1.5]
      \node (G') at (0,0){$G'$}; \node (G) at (0,1){$G$};
      \path[->] (G) edge node[fill=white, font=\footnotesize] {$g$} (G');
      \node (Da) at (1,1) {$D_a$}; \node (Da') at (1,0){$D'_a$};
      \path[->] (Da) edge node[fill=white, font=\footnotesize] {$f_a$} (G); 
      \path[->] (Da') edge node[fill=white, font=\footnotesize] {$f'_a$} (G'); 
      \path[->] (Da) edge node[fill=white, font=\footnotesize] {$t_a$} (Da'); 
      \node (C) at (2,1){$C$}; \node (C') at (2,0){$C'$};
      \path[->] (C) edge node[fill=white, font=\footnotesize] {$e_a$} (Da); 
      \path[->] (C') edge node[fill=white, font=\footnotesize] {$e'_a$} (Da'); 
      \path[->] (C) edge node[fill=white, font=\footnotesize] {$c$} (C'); 
      \node at (0.5,0.5){$(1_a)$}; \node at (1.5,0.5){$(2_a)$};
    \end{tikzpicture}
  \end{center}
  if $(1_a)$ and $(2_a)$ are pushouts for all $1\leq a\leq p$,
  and the diagram
  \begin{center}
    \begin{tikzpicture}[xscale=2]
      \node (G) at (0,0) {$G'$};
      \node (D1) at (1,1) {$D'_1$};
      \node (Dp) at (1,-1) {$D'_p$};
      \node (C) at (2,0) {$C'$};
      \node at (1,0.1) {$\vdots$};
      \path[->] (D1) edge node[fill=white, font=\footnotesize] {$f'_1$} (G);
      \path[->] (Dp) edge node[fill=white, font=\footnotesize] {$f'_p$} (G);
      \path[->] (C) edge node[fill=white, font=\footnotesize] {$e'_1$} (D1);
      \path[->] (C) edge node[fill=white, font=\footnotesize] {$e'_p$} (Dp);
    \end{tikzpicture}
  \end{center}
  commutes then it is a limit. 
\end{lemma}
\begin{proof}
  Induction on $p$. For $p=1$, since $(D_1,\id{D_1})$ is a limit of
  $(f_1,G)$ then $e_1$ is an isomorphism, hence $e_1\in\M$ and
  $e'_1\in\M$. It is easy to see that $e'_1\circ(\id{C'},c\circ
  \invf{e_1},C') = (e'_1,t_1,D'_1)$, an extremal epi-sink since
  $(2_1)$ is a pushout. Hence $e'_1$ is an isomorphism and
  $(C',e'_1)$ is therefore a limit of $(f'_1,G')$.

  Assume that the property is true for $p-1\geq 1$, let
  $(C,e_1,\dotsc,e_p)$ be a limit of $(f_1,\dotsc,f_p,G)$ with
  $f_a, e_a\in\M$ and pushouts $(1_a)$ and $(2_a)$ such that
  $f'_a\circ e'_a = f'_1\circ e'_1$ for all $1\leq a\leq p$.  By
  Lemma~\ref{lm-indstep-Mlimit} there exists a limit
  $(B,g_2,\dotsc,g_p)$ of $(f_2,\dotsc,f_p,G)$ with
  $g_2,\dotsc,g_p\in\M$. Hence by Lemma~\ref{lm-indstep-limit} there
  exists a morphism $x:C\rightarrow B$ such that $e_a=g_a\circ x$ for
  all $2\leq a\leq p$ and
  \begin{center}
    \begin{tikzpicture}[xscale=-2.5,yscale=1.5]
      \node (B) at (0,0) {$B$}; \node (C) at (0,1){$C$}; \node (G) at
      (1,0) {$G$}; \node (D1) at (1,1) {$D_1$}; \node at (0.5,0.5){(1)};
      \path[->] (C) edge node[fill=white, font=\footnotesize] {$x$}(B); 
      \path[->] (C) edge node[fill=white, font=\footnotesize] {$e_1$}(D1); 
      \path[->] (D1) edge node[fill=white, font=\footnotesize] {$f_1$}(G); 
      \path[->] (B) edge node[fill=white, font=\footnotesize] {$f_2\circ g_2$}(G); 
    \end{tikzpicture}
  \end{center}
  is a pullback. But $f_1\in\M$, hence $x\in\M$ and therefore there is
  a pushout $(x',b,B')$ of $(C,x,c)$ (square (2) below) and
  $x'\in\M$.
  \begin{center}
    \begin{tikzpicture}[xscale=2,yscale=1.5]
      \node (G') at (0,0){$G'$}; \node (G) at (0,1){$G$};
      \path[->] (G) edge node[fill=white, font=\footnotesize] {$g$} (G');
      \node (Da) at (1,1) {$D_a$}; \node (Da') at (1,0){$D'_a$};
      \path[->] (Da) edge node[fill=white, font=\footnotesize] {$f_a$} (G); 
      \path[->] (Da') edge node[fill=white, font=\footnotesize] {$f'_a$} (G'); 
      \path[->] (Da) edge node[fill=white, font=\footnotesize] {$t_a$} (Da'); 
      \node (B) at (2,1){$B$}; \node (B') at (2,0){$B'$};
      \path[->] (B) edge node[fill=white, font=\footnotesize] {$g_a$}(Da);
      \path[->] (B) edge node[fill=white, font=\footnotesize] {$b$}(B');
      \path[->,dashed] (B') edge node[fill=white, font=\footnotesize] {$g'_a$}(Da');
      \node (C) at (3,1){$C$}; \node (C') at (3,0){$C'$};
      \path[->] (C) edge node[fill=white, font=\footnotesize] {$x$}(B);
      \path[->] (C') edge node[fill=white, font=\footnotesize] {$x'$}(B');
      \path[->] (C) edge [bend right] node[fill=white, font=\footnotesize] {$e_a$} (Da); 
      \path[->] (C') edge [bend left] node[fill=white, font=\footnotesize] {$e'_a$} (Da'); 
      \path[->] (C) edge node[fill=white, font=\footnotesize] {$c$} (C'); 
      \node at (0.5,0.5){$(1_a)$}; \node at (1.5,0.5){$(3_a)$}; \node at (2.5,0.5){$(2)$};
    \end{tikzpicture}
  \end{center}
  Since (2) is a pushout and
  $e'_a\circ c=t_1\circ e_a = t_a\circ g_a\circ x$ for all
  $2\leq a\leq p$ then there exists a unique morphism
  $g'_a:B'\rightarrow D'_a$ such that $g'_a\circ x'=e'_a$ and
  $g'_a\circ b=t_a\circ g_a$. We thus have
  $f'_a\circ g'_a\circ x' = f'_a\circ e'_a = f'_2\circ e'_2 =
  f'_2\circ g'_2\circ x'$
  and
  $f'_a\circ g'_a\circ b = g\circ f_a\circ g_a = g\circ f_2\circ g_2 =
  f'_2\circ g'_2\circ b$,
  and since $(x',b,B')$ is an epi-sink then
  $f'_a\circ g'_a = f'_2\circ g'_2$. Besides, since $(2)+(3_a)=(2_a)$
  is a pushout then by decomposition $(3_a)$ is a pushout.  We may
  thus apply the induction hypothesis to the limit
  $(B,g_2,\dotsc,g_p)$ and the pushouts $(1_a)$ and $(3_a)$, which
  yields that $(B',g'_2,\dotsc,g'_p)$ is a limit of
  $(f'_2,\dotsc,f'_p,G')$. 

  We now consider the cube
  \begin{center}
    \begin{tikzpicture}[scale=0.65]
  \cubenodes{G'}{B'}{D'_1}{C'}{G}{B}{D_1}{C};
  \path[<-] (LL) edge node[fill=white, font=\footnotesize] {$f'_2\circ g'_2$} (LF);
  \path[<-] (LF) edge node[fill=white, font=\footnotesize] {$x'$} (LR);
  \path[->] (LR) edge node[fill=white, font=\footnotesize] {$e'_1$} (LB);
  \path[->] (LB) edge node[fill=white, font=\footnotesize] {$f'_1$} (LL);
  \path[<-] (LL) edge node[fill=white, font=\footnotesize] {$g$} (UL); 
  \path[-] (LF) edge [draw=white, line width=3pt] (UF); 
  \path[<-] (LF) edge node[near start,fill=white, font=\footnotesize] {$b$} (UF); 
  \path[<-] (LR) edge node[fill=white, font=\footnotesize] {$c$} (UR); 
  \path[<-] (LB) edge node[near end,fill=white, font=\footnotesize] {$t_1$} (UB); 
  \path[-] (UL) edge [draw=white, line width=3pt] (UF);
  \path[<-] (UL) edge node[fill=white, font=\footnotesize] {$f_2\circ g_2$} (UF);
  \path[<-] (UF) edge node[fill=white, font=\footnotesize] {$x$} (UR);
  \path[->] (UR) edge node[fill=white, font=\footnotesize] {$e_1$} (UB);
  \path[->] (UB) edge node[fill=white, font=\footnotesize] {$f_1$} (UL);
  \end{tikzpicture}
  \end{center}
  It is obvious that all horizontal morphisms are in $\M$. The top
  face is the pullback (1), the vertical faces are the pushouts
  $(1_2)+(3_2)$, $(1_1)$, (2) and $(2_1)$. The bottom face commutes
  since $f'_1\circ e'_1 = f'_2\circ e'_2 = f'_2\circ g'_2\circ x'$,
  hence the cube commutes and we may apply the cube POPB lemma, which
  yields that the bottom face is a pullback. Hence by
  Lemma~\ref{lm-indstep-limit} $(C', e'_1,\dotsc,e'_p)$ is a limit
  of $(f'_1,\dotsc, f'_p, G')$, which completes the induction. 
\end{proof}

Note that this lemma generalizes the if part of the cube POPB lemma to
the case where the top ``face'' is a limit diagram. A similar result
holds for colimits, that is not restricted to adhesive HLR
categories.

\begin{lemma}\label{lm-shift-colimit}
  For all $p\geq 1$, source $(C,s_1,\dotsc,s_p)$, colimit $(h_1,\dotsc,h_p,H)$ of
  $(C,s_1,\dotsc,s_p)$, morphisms $h:H\rightarrow H'$, $c:C\rightarrow
  C'$ and $t'_a,s'_a,h'_a$ below,
  \begin{center}
    \begin{tikzpicture}[xscale=-2,yscale=1.5]
      \node (G') at (0,0){$H'$}; \node (G) at (0,1){$H$};
      \path[->] (G) edge node[fill=white, font=\footnotesize] {$h$} (G');
      \node (Da) at (1,1) {$F_a$}; \node (Da') at (1,0){$F'_a$};
      \path[->] (Da) edge node[fill=white, font=\footnotesize] {$h_a$} (G); 
      \path[->] (Da') edge node[fill=white, font=\footnotesize] {$h'_a$} (G'); 
      \path[->] (Da) edge node[fill=white, font=\footnotesize] {$t'_a$} (Da'); 
      \node (C) at (2,1){$C$}; \node (C') at (2,0){$C'$};
      \path[->] (C) edge node[fill=white, font=\footnotesize] {$s_a$} (Da); 
      \path[->] (C') edge node[fill=white, font=\footnotesize] {$s'_a$} (Da'); 
      \path[->] (C) edge node[fill=white, font=\footnotesize] {$c$} (C'); 
      \node at (0.5,0.5){$(4_a)$}; \node at (1.5,0.5){$(3_a)$};
    \end{tikzpicture}
  \end{center}
  if $(3_a)$ and $(4_a)$ are pushouts for all $1\leq a\leq p$,
  and the diagram
  \begin{center}
    \begin{tikzpicture}[xscale=-2]
      \node (G) at (0,0) {$H'$};
      \node (D1) at (1,1) {$F'_1$};
      \node (Dp) at (1,-1) {$F'_p$};
      \node (C) at (2,0) {$C'$};
      \node at (1,0.1) {$\vdots$};
      \path[->] (D1) edge node[fill=white, font=\footnotesize] {$h'_1$} (G);
      \path[->] (Dp) edge node[fill=white, font=\footnotesize] {$h'_p$} (G);
      \path[->] (C) edge node[fill=white, font=\footnotesize] {$s'_1$} (D1);
      \path[->] (C) edge node[fill=white, font=\footnotesize] {$s'_p$} (Dp);
    \end{tikzpicture}
  \end{center}
  commutes then it is a colimit. 
\end{lemma}
\begin{proof}
  Let $(z_1,\dotsc,z_p,X)$ such that $z_a\circ s'_a = z_1\circ s'_1$
  for all $1\leq a\leq p$. Then
  $z_a\circ t'_a\circ s_a = z_a\circ s'_a\circ c = z_1\circ s'_1 \circ
  c = z_1\circ t'_1\circ s_1$, hence there exists a unique morphism
  $y:H\rightarrow X$ such that $y\circ h_a = z_a\circ t'_a$ for all
  $1\leq a\leq p$. We thus have $z_1\circ s'_1\circ c = y\circ
  h_1\circ s_1$, and since by composition $(3_1)+(4_1)$ is a pushout
  then there exists a unique morphism $x:H'\rightarrow X$ such that
  $z_1\circ s'_1=x\circ h'_1\circ s'_1$ and $y=x\circ h$.

  Since $(4_a)$ is a pushout then $y\circ h_a = z_a\circ t'_a$ yields
  a unique morphism $x_a:H'\rightarrow X$ such that $z_a=x_a\circ h'_a$
  and $y=x_a\circ h$. But then we have $x_a\circ h'_1\circ s'_1 =
  x_a\circ h'_a\circ s'_a = z_a\circ s'_a = z_1\circ s'_1$, hence by
  unicity of $x$ we have $x_a=x$, which yields $x\circ h'_a = z_a$ for
  all $1\leq a\leq p$, which proves that $(h'_1,\dotsc,h'_p,H')$ is a
  colimit of $(C',s'_1,\dotsc,s'_p)$. 
\end{proof}

\begin{theorem}
  For any set $\R$ of $\M$-weak spans and any objects $G$, $H$, then
  any parallel coherent transformation $G\ruletrans{\R}H$ has a
  derived rule. 
\end{theorem}
 \begin{proof}
  Let $\Gamma$ be a parallel coherent diagram for $G$ and $\pi$ be a
  PCT of $G$ by $\Gamma$, as in Figure~\ref{fig-pct}, and $\aspan$ be the
  span $G\xleftarrow{l} C \xrightarrow{r} H$ where $l=f_1\circ e_1$
  and $r=h_1\circ s_1$, so that $G\trans{\aspan}H$. Assume that
  $G'\trans{\aspan}H'$, then there is a diagram
  \begin{center}
      \begin{tikzpicture}[scale=1.5]
        \node (G) at (0,0) {$G'$}; \node (L) at (0,1) {$G$};
        \node (K) at (1,1) {$C$}; \node (D) at (1,0) {$C'$};
        \node (RK) at (2,1) {$H$}; \node (H) at (2,0) {$H'$}; 
        \node at (0.5,0.5) {(1)}; \node at (1.5,0.5) {(2)};
        \path[->] (K) edge node[fill=white, font=\footnotesize] {$l$} (L);
        \path[->] (L) edge node[fill=white, font=\footnotesize] {$g$} (G);
        \path[->] (K) edge node[fill=white, font=\footnotesize] {$c$} (D);
        \path[->] (D) edge node[fill=white, font=\footnotesize] {$f$} (G);
        \path[->] (D) edge node[fill=white, font=\footnotesize] {$g$} (H);
        \path[->] (K) edge node[fill=white, font=\footnotesize] {$r$} (RK);
        \path[->] (RK) edge node[fill=white, font=\footnotesize] {$h$} (H);
      \end{tikzpicture}
  \end{center}
  where (1) and (2) are pushouts. Obviously $f_a,e_a,s_a,h_a\in\M$ for
  all $1\leq a\leq p$, hence there exists a pushout $(e'_a, t_a,D'_a)$
  of $(C,e_a,c)$, the square $(2_a)$ below
  \begin{center}
    \begin{tikzpicture}[xscale=2,yscale=1.5]
      \node (G') at (0,0){$G'$}; \node (G) at (0,1){$G$};
      \path[->] (G) edge node[fill=white, font=\footnotesize] {$g$} (G');
      \node (Da) at (1,1) {$D_a$}; \node (Da') at (1,0){$D'_a$};
      \path[->] (Da) edge node[fill=white, font=\footnotesize] {$f_a$} (G); 
      \path[->,dashed] (Da') edge node[fill=white, font=\footnotesize] {$f'_a$} (G'); 
      \path[->] (Da) edge node[fill=white, font=\footnotesize] {$t_a$} (Da'); 
      \node (C) at (2,1){$C$}; \node (C') at (2,0){$C'$};
      \path[->] (C) edge node[fill=white, font=\footnotesize] {$e_a$} (Da); 
      \path[->] (C') edge node[fill=white, font=\footnotesize] {$e'_a$} (Da'); 
      \path[->] (C) edge node[fill=white, font=\footnotesize] {$c$} (C'); 
      \path[->,bend right] (C) edge node[fill=white, font=\footnotesize] {$l$} (G); 
      \path[->,bend left] (C') edge node[fill=white, font=\footnotesize] {$f$} (G'); 
      \node at (0.5,0.5){$(1_a)$}; \node at (1.5,0.5){$(2_a)$};
    \end{tikzpicture}
  \end{center}
  and since (1) commutes then there is a unique morphism
  $f'_a:D'_a\rightarrow G'$ such that $f'_a\circ e'_a = f$ and
  $f'_a\circ t_a=g\circ f_a$, and by decomposition $(1_a)$ is a
  pushout. Hence by Lemma~\ref{lm-shift-limit} $(C',e'_1,\dotsc,e'_p)$
  is a limit of $(f'_1,\dotsc,f'_p,G')$.

  Similarly, there exists a pushout $(s'_a, t'_a,F'_a)$
  of $(C,s_a,c)$, the square $(3_a)$ below
  \begin{center}
    \begin{tikzpicture}[xscale=-2,yscale=1.5]
      \node (G') at (0,0){$H'$}; \node (G) at (0,1){$H$};
      \path[->] (G) edge node[fill=white, font=\footnotesize] {$h$} (G');
      \node (Da) at (1,1) {$F_a$}; \node (Da') at (1,0){$F'_a$};
      \path[->] (Da) edge node[fill=white, font=\footnotesize] {$h_a$} (G); 
      \path[->,dashed] (Da') edge node[fill=white, font=\footnotesize] {$h'_a$} (G'); 
      \path[->] (Da) edge node[fill=white, font=\footnotesize] {$t'_a$} (Da'); 
      \node (C) at (2,1){$C$}; \node (C') at (2,0){$C'$};
      \path[->] (C) edge node[fill=white, font=\footnotesize] {$s_a$} (Da); 
      \path[->] (C') edge node[fill=white, font=\footnotesize] {$s'_a$} (Da'); 
      \path[->] (C) edge node[fill=white, font=\footnotesize] {$c$} (C'); 
      \path[->,bend right] (C) edge node[fill=white, font=\footnotesize] {$r$} (G); 
      \path[->,bend left] (C') edge node[fill=white, font=\footnotesize] {$g$} (G'); 
      \node at (0.5,0.5){$(4_a)$}; \node at (1.5,0.5){$(3_a)$};
    \end{tikzpicture}
  \end{center}
  and since (2) commutes then there is a unique morphism
  $h'_a:F'_a\rightarrow H'$ such that $h'_a\circ s'_a = g$ and
  $h'_a\circ t'_a=h\circ h_a$, and by decomposition $(4_a)$ is a
  pushout. Hence by Lemma~\ref{lm-shift-colimit} $(h'_1,\dotsc,h'_p,H')$ is a
  colimit of $(C',s'_1,\dotsc,s'_p)$. 

  By hypothesis for all $1\leq a\leq p$  there is a
  $\adt_a\in\dirtranspo{G}{\arule_a}$ for some $r_a\in\R$, pictured below.
  \begin{center}
    \begin{tikzpicture}[xscale=2, yscale=1.5]
      \node (L) at (0,1) {$L_a$}; \node (K) at (1,1) {$K_a$};  \node (I) at
      (2,1) {$I_a$};  \node (R) at (3,1) {$R_a$}; \node (G) at (0.5,0) {$G$};
      \node (D) at (1.5,0) {$D_a$}; \node (H) at (2.5,0) {$H_a$};
      \node at (0.75,0.5) {\PO}; \node at (2.25,0.5) {\PO};
      \node at (-0.5,0.5) {$(\adt_a)$}; \node (G') at (0.5,-1) {$G'$};
      \node (D') at (1.5,-1) {$D'_a$}; \node (H') at (2.5,-1) {$H'_a$};
      \node at (1,-0.5) {$(1_a)$}; \node at (2,-0.5) {\PO};
      \path[->] (K) edge node[fill=white, font=\footnotesize] {$l_a$} (L);
      \path[->] (I) edge node[fill=white, font=\footnotesize] {$i_a$} (K);
      \path[->] (I) edge node[fill=white, font=\footnotesize] {$r_a$} (R);
      \path[->] (L) edge node[fill=white, font=\footnotesize] {$m_a$} (G);
      \path[->] (K) edge node[fill=white, font=\footnotesize] {$k_a$} (D);
      \path[->] (D) edge node[fill=white, font=\footnotesize] {$f_a$} (G);
      \path[->] (I) edge node[fill=white, font=\footnotesize] {$k_a\circ i_a$} (D);
      \path[->] (D) edge node[fill=white, font=\footnotesize] {$g_a$} (H);
      \path[->] (R) edge node[fill=white, font=\footnotesize] {$n_a$} (H);
      \path[->] (G) edge node[fill=white, font=\footnotesize] {$g$} (G');
      \path[->] (D) edge node[fill=white, font=\footnotesize] {$t_a$} (D');
      \path[->] (H) edge node[fill=white, font=\footnotesize] {$n'_a$} (H');
      \path[->] (D') edge node[fill=white, font=\footnotesize] {$f'_a$} (G');
      \path[->] (D') edge node[fill=white, font=\footnotesize] {$g'_a$} (H');
      \path[->,bend right=25] (L) edge node[fill=white, font=\footnotesize] {$g\circ m_a$} (G');      
      \path[->,bend left=25] (R) edge node[fill=white, font=\footnotesize] {$n'_a\circ n_a$} (H');      
      \node at (3.8,0) {$(\adt'_a)$};
    \end{tikzpicture}
  \end{center}
  Since $r_a\in\M$ then $g_a\in\M$ hence there is a pushout
  $(g'_a,n'_a,H'_a)$ of $(D_a,g_a,t_a)$, hence by pushout
  composition there is an obvious $\adt'_a\in\dirtranspo{G'}{\arule_a}$.

  Since $\Gamma$ is a parallel coherent diagram there
  exist morphisms $j_a^b: I_a\rightarrow D_b$ such that
  $f_b\circ j_a^b = f_1\circ j_a^1$ for all integers
  $1\leq a,b\leq p$. Let
  $j_a^{\prime b} = t_b\circ j_a^b: I_a\rightarrow D'_b$, then
  \[f'_b\circ j_a^{\prime b} = g\circ f_b\circ j_a^b = g\circ f_1\circ j_a^1 =
    f'_1\circ j_a^{\prime 1},\] hence the diagram $\Gamma'$ constituted
  of the diagrams $\adt'_1,\dotsc,\adt'_p$ and the morphisms
  $j_a^{\prime b}$ is a parallel coherent diagram. By the property of
  $C'$ there exists a unique morphism
  $d'_a:I_a\rightarrow C'$ such that $j_a^{\prime b} = e'_b\circ
  d'_a$. But from $\pi$ we have morphisms $d_a:I_a\rightarrow C$ such
  that $j_a^{b} = e_b\circ d_a$, and we see that $e'_b\circ c\circ d_a
  = t_b\circ e_b\circ d_a = t_b\circ j_a^b = j_a^{\prime b}$, hence by
  unicity $d'_a = c\circ d_a$.

  Finally, we see by pushout composition that
  $(s'_a,t'_a\circ o_a, F'_a)$ is a pushout of $(I_a, r_a, d'_a)$,
  and hence that the diagram
  
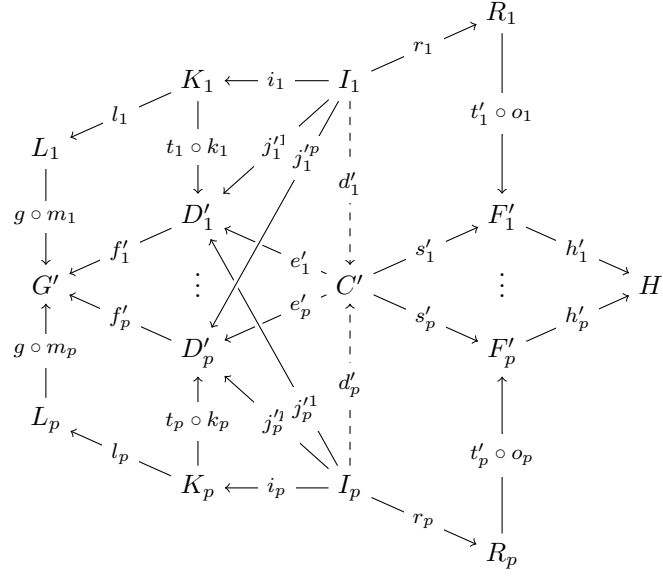
\begin{center}
          \begin{tikzpicture}[xscale=2, yscale=0.9]
  \node (G) at (0,0) {$G'$};
  \node (D) at (2,0) {$C'$};
  \node (H) at (4,0) {$H'$};
  \node at (1,0.1) {$\vdots$};
  \node at (3,0.1) {$\vdots$};
  \node (L1) at (0,2) {$L_1$};
  \node (K1) at (1,3) {$K_1$};
  \node (D1) at (1,1) {$D'_1$};
  \node (I1) at (2,3) {$I_1$};
  \node (R1) at (3,4) {$R_1$};
  \node (H1) at (3,1) {$F'_1$};
  \node (Ln) at (0,-2) {$L_p$};
  \node (Kn) at (1,-3) {$K_p$};
  \node (Dn) at (1,-1) {$D'_p$};
  \node (In) at (2,-3) {$I_p$};
  \node (Rn) at (3,-4) {$R_p$};
  \node (Hn) at (3,-1) {$F'_p$};
  \path[->] (K1) edge node[fill=white, font=\footnotesize] {$l_1$} (L1) ;
  \path[->] (L1) edge node[fill=white, font=\footnotesize] {$g\circ m_1$} (G);
  \path[->] (K1) edge node[fill=white, font=\footnotesize] {$t_1\circ k_1$} (D1);
  \path[->] (D1) edge node[fill=white, font=\footnotesize] {$f'_1$} (G);
  \path[->] (I1) edge node[fill=white, font=\footnotesize] {$i_1$} (K1);
  \path[->] (I1) edge node[fill=white, font=\footnotesize] {$r_1$} (R1);
  \path[->] (R1) edge node[fill=white, font=\footnotesize] {$t'_1\circ o_1$} (H1);
  \path[->] (D) edge node[fill=white, font=\footnotesize, near start] {$e'_1$} (D1);
  \path[->] (D) edge node[fill=white, font=\footnotesize] {$s'_1$} (H1);
  \path[->] (H1) edge node[fill=white, font=\footnotesize] {$h'_1$} (H);
  \path[->,dashed] (I1) edge node[fill=white, font=\footnotesize] {$d'_1$} (D);
  \path[->] (Kn) edge node[fill=white, font=\footnotesize] {$l_p$} (Ln) ;
  \path[->] (Ln) edge node[fill=white, font=\footnotesize] {$g\circ m_p$} (G);
  \path[->] (Kn) edge node[fill=white, font=\footnotesize] {$t_p\circ k_p$} (Dn);
  \path[->] (Dn) edge node[fill=white, font=\footnotesize] {$f'_p$} (G);
  \path[->] (In) edge node[fill=white, font=\footnotesize] {$i_p$} (Kn);
  \path[->] (In) edge node[fill=white, font=\footnotesize] {$r_p$} (Rn);
  \path[->] (Rn) edge node[fill=white, font=\footnotesize] {$t'_p\circ o_p$} (Hn);
  \path[->] (D) edge node[fill=white, font=\footnotesize, near start] {$e'_p$} (Dn);
  \path[->] (D) edge node[fill=white, font=\footnotesize] {$s'_p$} (Hn);
  \path[->] (Hn) edge node[fill=white, font=\footnotesize] {$h'_p$} (H);
  \path[->,dashed] (In) edge node[fill=white, font=\footnotesize] {$d'_p$} (D);
  \path[->] (I1) edge node[fill=white, font=\footnotesize] {$j^{\prime 1}_1$} (D1);
  \path[->] (In) edge node[fill=white, font=\footnotesize] {$j^{\prime p}_p$} (Dn);
  \path[-] (In) edge[draw=white, line width=3pt]  (D1);
  \path[->] (In) edge node[fill=white, font=\footnotesize, near start] {$j_p^{\prime 1}$} (D1);
  \path[-] (I1) edge[draw=white, line width=3pt]  (Dn);
  \path[->] (I1) edge node[fill=white, font=\footnotesize, near start] {$j^{\prime p}_1$} (Dn);
\end{tikzpicture}
  \end{center}
  is a PCT of $G'$ by $\Gamma'$. This proves that
  $G'\trans{\Gamma'}H'$ and hence that $G'\ruletrans{\R}H'$.  
\end{proof}


\begin{thebibliography}{1}

\bibitem{AdamekHS04}
Ji\v{r}{\'i} Ad{\'a}mek, Horst Herrlich, and George~E. Strecker.
\newblock {\em Abstract and Concrete Categories - The Joy of Cats}.
\newblock Online Edition, 2004.
\newblock URL: \url{http://katmat.math.uni-bremen.de/acc/}.

\bibitem{BoydelatE19}
Thierry Boy de~la Tour and Rachid Echahed.
\newblock True parallel graph transformations: an algebraic approach based on
  weak spans.
\newblock {\em CoRR}, abs/1904.08850, 2019.
\newblock URL: \url{http://arxiv.org/abs/1904.08850}, \href
  {http://arxiv.org/abs/1904.08850} {\path{arXiv:1904.08850}}.

\bibitem{EhrigEPT06}
Hartmut Ehrig, Karsten Ehrig, Ulrike Prange, and Gabriele Taentzer.
\newblock {\em Fundamentals of Algebraic Graph Transformation}.
\newblock Monographs in Theoretical Computer Science. An {EATCS} Series.
  Springer, 2006.
\newblock \href {http://dx.doi.org/10.1007/3-540-31188-2}
  {\path{doi:10.1007/3-540-31188-2}}.

\end{thebibliography}

\end{document}